\documentclass[11pt]{amsart}
\usepackage{amsfonts}
\usepackage[arrow,matrix]{xy}
\usepackage{amsmath,amssymb, bbm, amscd, amsthm, mathrsfs, hyperref, mathtools}
\allowdisplaybreaks[2]   
\usepackage[titletoc]{appendix}
\usepackage{tikz-cd}
\usepackage{bm}
\usepackage{eqparbox}

\DeclareMathAlphabet{\mathpzc}{OT1}{pzc}{m}{it}

\theoremstyle{plain}
\newtheorem{thm}{Theorem}[section]
\newtheorem{cor}[thm]{Corollary}
\newtheorem{lem}[thm]{Lemma}
\newtheorem{prop}[thm]{Proposition}

\theoremstyle{definition}
\newtheorem{defn}[thm]{Definition}

\newtheorem{notation}[thm]{Notation}
\newtheorem{rem}[thm]{Remark}
\newtheorem{con}[thm]{Convention}
\newtheorem{hypo}[thm]{Hypothesis}

\numberwithin{equation}{section}

\def\mc{\mathcal}
\def\mk{\mathfrak}

\def\ot{\otimes}
\def\op{\oplus}
\def\le{\leqslant}
\def\ge{\geqslant}

\def\ra{\rightarrow}

\def\Hom{\operatorname {Hom}}
\def\Ext{\operatorname {Ext}}

\def\dim{\operatorname {dim}}
\def\id{\operatorname {id}}

\def\GL{\operatorname {GL}}

\def\tr{\operatorname {tr}}
\def\hdet{\operatorname {hdet}}
\def\Aut{\operatorname {Aut}}
\def\GrAut{\operatorname {GrAut}}
\def\Sym{\operatorname {Sym}}
\def\Tr{\operatorname {Tr}}
\def\GKdim{\operatorname {GKdim}}

\def\kk{\mathbbm{k}}
\def\ZZ{\mathbb{Z}}
\def\NN{\mathbb{N}}

\def\trv{\mathsf{t}}

\begin{document}
	\title[Group actions on Yang-Baxter algebras]{\bf Yang-Baxter permutation group actions on distributive Yang-Baxter algebras}

	\author{Ji-Wei HE}
	\address {Ji-Wei HE\newline School of Mathematics, Hangzhou Normal University, Hangzhou, Zhejiang 310036, China}
	
	\email{jwhe@hznu.edu.cn}
	
	\author{Xiaolan YU}
	\address {Xiaolan YU\newline School of Mathematics, Hangzhou Normal University, Hangzhou, Zhejiang 310036, China}
	
	\email{xlyu@hznu.edu.cn}

	\date{}

	\begin{abstract}
		
		Let $(X,r)$ be a finite square-free distributive
		set-theoretic solution of the Yang-Baxter equation,
		and let $\mathcal{A}(X,r)$ be its associated Yang-Baxter
		algebra. We prove that $\mathcal{A}(X,r)$ is isomorphic
		to a skew polynomial algebra and explicitly compute its
		Nakayama automorphism. We study the natural faithful
		action of the permutation group $\mathcal{G}(X,r)$ on
		$\mathcal{A}(X,r)$ by graded algebra automorphisms.
		We characterize when $\mathcal{G}(X,r)$ is a reflection
		group and, in that case, describe the invariant subalgebra
		$\mathcal{A}(X,r)^{\mathcal{G}(X,r)}$ together with its
		Jacobian, reflection arrangement, and discriminant.
		We further establish a necessary and sufficient
		condition for the noncommutative Auslander theorem to
		hold for the permutation group action on a distributive
		Yang-Baxter algebra. Finally,
		for a class of non-trivial distributive Yang-Baxter
		algebras, we obtain a lower bound for the pertinency
		of the permutation group action.
		
	\end{abstract}

	
	\maketitle
	
	\section*{Introduction}
	
	The Yang-Baxter equation (YBE) is a fundamental equation in mathematical physics and representation theory, arising originally from the study of exactly solvable models in statistical mechanics \cite{bax,yang}. It laid the foundations of the theory of quantum groups and Hopf algebras (see for example \cite{bg1,kas}). In a seminal work, Drinfeld \cite{dr92} proposed the study of set-theoretical solutions to the YBE, i.e., bijective maps \(r\colon X\times X\to X\times X\) on a set \(X\) satisfying the braid relation
	\[
	(r\times \mathrm{id}_X)(\mathrm{id}_X\times r)(r\times \mathrm{id}_X)
	=(\mathrm{id}_X\times r)(r\times \mathrm{id}_X)(\mathrm{id}_X\times r)
	\]
	on \(X\times X\times X\). This line of research has since flourished, leading to systematic classifications and deep connections with braid groups, semigroups of I-type, and noncommutative algebra \cite{cho,ess,giv,lyz,ru07}. Further developments revealed links to rack and quandle theory, enriching the interplay between set-theoretical solutions and low-dimensional topology \cite{br,cjks,fr}.
	
	To each set-theoretical solution \((X, r)\), one can associate an associative algebra, often called the Yang-Baxter (YB) algebra \(\mathcal{A}(X, r)\) (in the sense of \cite{man}). Typically defined as the quotient of the free algebra \(\kk\langle X \rangle\) by relations encoding the map \(r\), this algebra serves as a ``noncommutative coordinate ring'' of the solution. Understanding its algebraic structure is thus a natural step,  as the algebra
	reflects the combinatorial dynamics of \(r\) through its homological and
	ring-theoretic properties. For instance, when the solution is square-free, non-degenerate and involutive, the corresponding algebra is known to be a quadratic algebra with good homological properties \cite{gi4,gi18,giv,jkv}. Especially, for finitely generated quantum binomial algebras \cite{gi6,gim1},
	such YB algebras are actually Artin-Schelter (AS-) regular \cite{gi2}, meaning that they share homological properties similar to those of commutative polynomial rings. In this paper we will mainly study this class of AS-regular algebras,
	namely the quantum binomial YB algebras.
	
	AS-regular algebras form a central topic in noncommutative algebras. Among the most important problems in this area is the study of group actions on AS-regular algebras and their invariant subalgebras, which provides noncommutative counterparts to classical invariant theory. Now recall that every set-theoretical solution \((X,r)\) gives rise to a  {permutation group} \(\mc{G}(X,r)\subseteq\operatorname{Sym}(X)\), namely the subgroup generated by the left actions \(L_x:y\mapsto y'\) where \(r(x,y)=(y',x')\). A natural question is: when does this permutation group act on the YB algebra \(\mathcal{A}(X,r)\) by graded algebra automorphisms? The answer turns out to be precisely when the solution is  {distributive}. In this case, the action of \(\mc{G}(X,r)\) on \(X\) extends linearly to an action on \(\mathcal{A}(X,r)\).
	
	We first investigate the structure of the YB algebra for a distributive solution. We prove that such an algebra is isomorphic to a skew polynomial algebra.  It can also be derived from \cite{gim2}. However, our proof  is more elementary and self-contained. We then compute its Nakayama
	automorphism  explicitly.
	
	The study of invariant subalgebras is a central theme
	in the theory of AS-regular algebras. We therefore
	consider the natural faithful action of the permutation
	group $\mc{G}(X,r)$ on $\mathcal{A}(X,r)$ by graded
	algebra automorphisms. We give a necessary and sufficient
	condition for $\mc{G}(X,r)$ to be a reflection group.
	In that case, we describe the invariant subalgebra
	$\mathcal{A}(X,r)^{\mc{G}(X,r)}$ and compute its Jacobian,
	reflection arrangement, and discriminant.
	
	Furthermore, we give a necessary and sufficient
	condition for the
	noncommutative Auslander theorem to hold for the
	permutation group action.	
	Originating in Auslander's work in commutative algebra
	and representation theory \cite{aus62}, this theorem
	asserts, under suitable hypotheses, that the natural
	map from the skew group algebra to the endomorphism
	ring of the algebra, viewed as a module over its
	invariant subalgebra, is an isomorphism. It plays an
	important role in the study of the McKay correspondence,
	isolated singularities, and related homological questions.
	
	Finally, we study pertinency, a numerical invariant
	of a group action introduced by Bao, Zhang, and the
	first-named author \cite{bhz, bhz1}. Under suitable
	hypotheses, pertinency provides a criterion for the
	Auslander map to be an isomorphism, but it is often
	difficult to compute explicitly. For a class of
	non-trivial distributive YB algebras, we obtain a lower
	bound for the pertinency of the permutation group action.

	\section{Preliminaries}
	Throughout $\kk$ is an algebraically closed field of characteristic zero. All vector spaces and algebras are over $\kk$.

	\subsection{Set-theoretic solution of  the Yang-Baxter equation} By a \emph{quadratic set}, we mean a pair $(X,r)$, where $X$ is a nonempty set and  $r:X\times X\ra X\times X$  a bijective map. We write the image of $(x,y)$ under $r$ as
	$$r(x, y)=\left({}^xy, x^y\right).$$
	The above equation defines a ``left action'' $\mc{L}:X\times X\ra X$, and a  ``right action'' $\mc{R}:X\times X\ra X$, on $X$ as:
	$$\mc{L}_x(y)={}^xy,\;\;\;\;\;  \mc{R}_y(x)=x^y,$$
	for all $x,y\in X$.
	
	A quadratic set $(X,r)$ is called:
	\begin{enumerate}
		\item 	\emph{non-degenerate} if the maps $\mc{L}_x$ and $\mc{R}_x$ are  bijections, for all  $x\in X$;
		\item \emph{involutive} if  $r^{2}={\id}_{X^{2}}$;
		\item \emph{square-free} if $ r(x, x)=(x, x)$, for every  $x \in X$;
		\item a \emph{quantum binomial set} if it is nondegenerate, involutive and square-free.
		
	\end{enumerate}

	A \emph{set-theoretic solution of the Yang-Baxter equation} (YBE) is a quadratic set $(X, r)$, such that the following braid relation holds
	$$(\id \times r)(r \times \id)(\id \times r)=(r \times \id)(\id \times r)(r \times \id).$$
	In this case, $(X,r)$ is also called a \emph{braided set}.
	A solution is called \emph{non-degenerate} (\emph{involutive}, or \emph{square-free}) if it is non-degenerate (involutive, or square-free) as a quadratic set.
	If $X$ is a finite set, then $(X,r)$ is called a finite solution.

	\begin{con}
		Throughout this paper, a \emph{solution} means a finite,
		non-degenerate, involutive set-theoretic solution of the YBE.  We assume that $|X|\ge2$.
		
	\end{con}

	Each quadratic set $(X,r)$ determines a set of quadratic relations $\bold{R}_0(r)$ defined by
	$$\begin{array}{l}
		x y=y' x' \in \bold{R}_0(r) \text{ if and only if }\\
		r(x,y)=(y',x') \text{ and } (x, y) \neq (y',x')  \text{ holds in } X\times X.
	\end{array}
	$$
	The monoid $S(X,r)=\langle X;\bold{R}_0(r) \rangle$, with a set of generators $X$ and a set of defining relations $\bold{R}_0(r)$, is called the \emph{semigroup associated with} $(X,r)$. The \emph{group} $G(X,r)$ \emph{associated with} $(X,r)$ is defined analogously.
	The \emph{algebra associated with}  $(X, r)$  is defined as
	$\mc{A}(X, r)=\kk\langle X\rangle /(\bold{R}(r))$, where
	$$\bold{R}(r)=\{xy-y'x'\mid xy=y'x'\in \bold{R}_0(r)\}.$$
	
	When $(X,r)$ is a solution of the YBE, $S(X,r)$, resp. $G(X,r)$, $\mc{A}(X, r)$  is called the \emph{Yang-Baxter} (YB) \emph{semigroup}, resp. the  \emph{YB group}, the  \emph{YB algebra}.
	
	\begin{defn}\cite{gi4}, \cite{gi5}
		A monoid $S$ is called of \textit{a monoid of skew-polynomial type}  if it has a standard finite presentation as
		\[
		S = \langle X; \bold{R}_0 \rangle,
		\]
		where the set of generators $X$ is ordered: $x_1 < x_2 < \cdots < x_n$, and $\bold{R}_0$ is a set of $\binom{n}{2}$ quadratic relations,
		\begin{equation}\label{e R}
			\bold{R}_0 = \{x_j x_i = x_{i'} x_{j'} \mid 1 \le i < j \le n,\ 1 \le i' < j' \le n\},
		\end{equation}
		satisfying
		\begin{itemize}
			\item[(i)] each monomial $xy \in X^2$, with $x \ne y$, occurs in exactly one relation in $\bold{R}_0$ (a monomial of the type $xx$ does not occur in any relation in $\bold{R}_0$);
			\item[(ii)] if $(x_j x_i = x_{i'} x_{j'}) \in \bold{R}_0$, with $1 \le i < j \le n$, then $i' < j'$, and $j > i'$;
			\item[(iii)] the monomials $x_k x_j x_i$ with $k > j > i$, $1 \le i, j, k \le n$ do not give rise to new relations in $S$.
		\end{itemize}
	\end{defn}
	
	\begin{rem}\label{rem0}Clearly, the set of relations (\ref{e R}) defines canonically an involutive bijective map
		\[\begin{array}{rcll}
			r: X \times X &\longrightarrow &X \times X,&\\
			
			(x_j, x_i) &\longleftrightarrow &(x_{i'}, x_{j'}), &1 \le i < j \le n, (i' < j', j > i', i < j');\\
			
			(x_i, x_i)& \longleftrightarrow& (x_i, x_i), & 1\le i\le n.\\
		\end{array}
		\]
		and $(X, r)$ is a quantum binomial quadratic set \cite{gi5}. 
		
		Consequently, for $x_i,x_j\in X$ and $i\neq j$, if $x_ix_j=x_sx_t$ in $\mc{A}(X,r)$ then either $i=s$ and $j=t$ or $r(x_i,x_j)=(x_s,x_t)$.
	\end{rem}
	
	The following lemma follows from \cite[Theorem 2.26]{gi6}.
	\begin{lem}\label{lem skew}
		Let $(X, r)$ be a square-free solution of the YBE, where $X$ is a finite quantum binomial set with $n$ elements, $n \ge 2$.
		Then there exists an ordering on $X$ such that
		$S = S(X, r)$ is a semigroup of skew-polynomial type.
		
	\end{lem}
	
	Let $(X,r)$ be a nondegenerate quadratic set. It is well-known that $(X,r)$ is a set-theoretic solution of the YBE if and only if the following condition holds for all $x,y,z\in X$ (e.g. \cite[Lemma 2.5]{gim1}).
	$$\mathbf{l1}:  {}^{x}\left(^{y} z\right)={ }^{^{x} y}\left(^{x^{y}} z\right), \quad \mathbf{r1}: \left(x^{y}\right)^{z}=\left(x^{^yz}\right)^{y^{z}}, \quad \mathbf{lr3}: \left({ }^{x} y\right)^{\left(^{x^{y}}(z)\right)}={ }^{(x{^{^y z}})}\left(y^{z}\right).$$
	
	For each non-degenerate braided set, condition {\bf l1} implies that  the map $x\ra \mc{L}_x$ extends canonically to a group homomorphism
	$$\mc{L}:G(X,r)\longrightarrow \Sym(X),$$
	where $\Sym(X)$ is symmetric group on $X$. The homomorphism $\mc{L}$ defines the \emph{canonical left action} of $G(X,r)$ on the set $X$. Similarly, {\bf r1} gives the \emph{canonical right action} of
	$G(X,r)$ on the set $X$.
	
	The  \emph{YB permutation group} $\mc{G}(X, r)$ of a finite solution $(X, r)$ is the subgroup of the symmetric
	group $\Sym(X)$ on $X$ generated by $\{\mc{L}_x \mid x\in X\}$.
	
	A quadratic set  $(X, r)$  is called \emph{cyclic} if the following conditions are satisfied
	$${\bf cl1}:  \quad {}^{y^{x}} x={ }^{y} x  \text{ for all }  x, y \in X; \quad  {\bf cr1}:  \quad x^{{}^x y}=x^{y}   \text{ for all } x,y\in X; $$
	$${\bf cl2}:  \quad {}^{{}^{x} y} x={ }^{y} x \text{ for all }   x, y \in X; \quad {\bf cr2}:  \quad x^{y^{x}}=x^{y}  \text{ for all }   x, y \in X.$$
	The above conditions are usually called  \emph{cyclic conditions}.
	Condition {\bf lri} is defined as
	
	${\bf lri} : \quad\quad\quad\quad\quad\quad\quad\quad\quad\quad\quad\left({ }^{x} y\right)^{x}=y={ }^{x}\left(y^{x}\right)$
	
	for all  $x, y \in X$. In other words {\bf lri} holds if and only if  $(X, r)$  is nondegenerate and  $\mathcal{R}_{x}=\mathcal{L}_{x}^{-1}$  and  $\mathcal{L}_{x}=\mathcal{R}_{x}^{-1}$ for all $x\in X$.
	
	By \cite[Theorem 2.35]{gim1}, every square-free solution is cyclic and satisfies {\bf lri}.

	Let $(X,r_X)$ and $(Y, r_Y)$ be two solutions of the YBE. A map $\phi:X\ra Y$ is a {homomorphism of solutions}, if it satisfies the equality
	$$(\phi\times \phi) r_X=r_Y (\phi\times \phi).$$
	If moreover, $\phi$ is a bijection, it is called an \emph{isomorphism}. An isomorphism of the solution $(X,r)$ onto itself is called an \emph{automorphism}. The group of automorphisms of $(X,r)$ is denoted by $\Aut(X,r)$. It is clear that $\Aut(X,r)$ is a subgroup of the symmetric group on $X$.
	
	The following useful remark  is straightforward from the definition.
	\begin{rem}
		\begin{enumerate}\label{rem auto}
			\item Let $(X, r_X)$ and $(Y, r_Y)$ be solutions. A map $\phi : X \longrightarrow Y$ is a homomorphism of solutions if and only if
			$\phi   \mc{L}_x = \mc{L}_{\phi(x)}  \phi$ and
			$\phi   \mc{R}_x = \mc{R}_{\phi(x)} \phi$ for all $x \in X$.
			\item  If both $(X, r_X)$ and $(Y, r_Y)$ satisfy the Condition \textbf{lri}, then $\phi$ is a homomorphism of solutions if and only if
			$\phi  \mc{L}_x = \mc{L}_{\phi(x)}   \phi$  for all $x \in X$.
			Especially,  if $(X, r)$ obeys the Condition \textbf{lri} (in particular, if $(X, r)$ is a square-free solution),  then $\sigma \in \Sym(X)$ is an automorphism of $(X, r_X)$ if and only if
			\begin{equation}\label{eq:auto_condition}
				\sigma   \mc{L}_x  \sigma^{-1} = \mc{L}_{\sigma(x)},
			\end{equation}
			for all $x \in X$.

		\end{enumerate}
		
	\end{rem}
	
	\begin{lem}\label{lem auto}
		Let $(X, r)$  be a square-free solution. A bijection $\phi$ on $X$ is an  automorphism of the solution $(X, r)$ if and only if it induces an automorphism of the YB algebra $\mc{A}(X,r)$.
	\end{lem}
	\proof It is clear that if $\phi$  is an  automorphism of the solution $(X, r)$, then it induces an automorphism of the  algebra $\mc{A}(X,r)$.
	
	Conversely, suppose the bijection $\phi$ induces an   automorphism of the YB algebra $\mc{A}(X,r)$. The solution $(X, r)$  is square-free, then for any $x\in X$,
	$$r(\phi(x),\phi(x))=(\phi(x),\phi(x))=({}^{\phi(x)}\phi(x),\phi(x)^{\phi(x)})$$
	holds. 	 Let $x,y\in X$ and $x\neq y$. In  the algebra  $\mc{A}(X,r)$, we have $xy=({}^xy )(x^y)$. Hence
	$$	\begin{array}{rcl}
		\phi(x)\phi(y)	&=&\phi(xy)\\
		&=&\phi(({}^xy )(x^y))\\
		&=&\phi({}^xy )\phi(x^y)
	\end{array}$$
	By Lemma \ref{lem skew}, the semigroup  $S(X;\bold{R}(r))$ is of skew-polynomial type. So $x\neq {}^xy$ and $y^x\neq y$, since $x\neq y$.  The map $\phi$ is a bijection, so $\phi(x)\neq \phi({}^xy)$ and $\phi(y^x)\neq \phi(y)$. With Remark \ref{rem0}, we obtain that $r(\phi(x),\phi(y))=(\phi({}^xy ), \phi(x^y))$. Therefore,
	$\phi({}x)={}^{\phi(x)}\phi(y)$ and $\phi(y^x)=\phi(y)^{\phi(x)}$. In conclusion,  the bijection $\phi$ is an   automorphism of the solution $(X,r)$. \qed

	\subsection{Artin-Schelter regular algebras}
	
	Unless explicitly indicated otherwise, by a \emph{graded algebra} we mean an $\mathbb{N}$-graded  algebra. A graded algebra  $A=\op_{i\in \NN} A_{i}$ is called  \emph{connected} if  $A_0=\kk$, and is called
	\emph{locally finite} if each $A_i$ is  finite-dimensional.
	The graded algebras we considered in this paper are all locally finite.
	
	The \emph{Hilbert series} of  a graded algebra $A$  is defined to be
	$$H_{A}(t)=\sum_{i =0}^{\infty}\left(\dim A_{i}\right) t^{i}.$$
	For a graded module $M =\op_{i\in \ZZ}M_i$, we write $M(n)$ for the $n$-th shift of $M$.

	\begin{defn}
		A connected graded algebra  $A$  is called \emph{Artin-Schelter (AS-for short) Gorenstein}   if
		\begin{enumerate}
			\item[(i)] $A$  has graded injective dimension $d<\infty$ on the left and on the right,
			\item[(ii)] $\Ext_{A}^{i}(\kk, A)=\Ext_{A^{op}}^{i}(\kk, A)=0$, for $i\neq d$,
			\item[(iii)] $\Ext_{A}^{d}(\kk, A)\cong \kk(l)$ and $\Ext_{A^{op}}^{d}(\kk, A)\cong \kk(l)$ for some $l\in \ZZ$.
		\end{enumerate}
		If in addition,
		\begin{enumerate}
			\item[(iv)] $A$  has finite (graded) global dimension  $d$,
			\item[(v)] $A$ has finite Gelfand-Kirillov (GK) dimension, meaning that the integer-valued function  $i \mapsto   \dim_{\kk} A_{i}$  is bounded by a polynomial in  $i$,
		\end{enumerate}
		then $A$ is called \emph{AS-regular} of dimension $d$. The number $l$ in the condition (iii) is called the \emph{AS-index} of $A$.
	\end{defn}

	The following Lemma can be extracted from \cite{gi1,gi2,giv,ru} (see also \cite[Remark 3.1]{gi}).
	\begin{lem}
		Let $(X,r)$ be a quantum binomial set and $\mc{A}=\mc{A}(X,r)$ be the associated quadratic algebra. The following conditions are equivalent.
		\begin{enumerate}
			\item $\mc{A}$ is an AS-regular PBW algebra.
			\item $(X,r)$ is a solution of the YBE.
		\end{enumerate}
		Each of these conditions implies that $\mc{A}$ is Koszul and a Noetherian domain.
	\end{lem}

	\subsection{Calabi-Yau algebras}
	For a graded algebra $A$, we denote by $A^{op}$ the opposite algebra of $A$ and $A^e$ the enveloping algebra $A\ot A^{op}$ of $A$. An $A$-bimodule can be identified with  a  left (right) $A^e$-module. We use $\GrAut(A)$ to denote the group of graded algebra automorphisms of $A$.
	
	For $\mu, \nu\in \GrAut(A)$  and  a graded left $A$-module $M$, we write ${}_\mu M$ for the   graded $A$-module  defined by ${}_\mu M\cong M$ as vector spaces, and
	$a\cdot m =\mu(a)m$, for all $a\in A$ and $m\in M$. Similarly, for a graded right  $A$-module $N$, we have the notion of $N_\mu$; and for a graded bimodule $L$, we have the notion ${}_\mu L_\nu$. If $\mu$ or $\nu$ is the identity, then we will omit it.  It is well-known that   $A_\mu\cong A$ as $A$-bimodules if and only
	if $\mu$ is an inner automorphism of $A$.
	
	Let $\mu\in \GrAut(A)$ and let $M,N$ be graded $A$-modules.
	A $\kk$-linear graded map $f: M \rightarrow N$ is called a $\mu$-linear homomorphism if it is a homomorphism of graded $A$-modules $f: M \rightarrow {}_{\mu}N$, i.e.\ if $f(am)=\mu(a)f(m)$ for all $a\in A$ and $m\in M$.
	
	\begin{defn}
		A graded algebra $A$ is called  a  \emph{twisted Calabi-Yau (CY for short) algebra  of dimension
			$d$} if
		\begin{enumerate}
			\item$A$ is \emph{homologically smooth}, that is,  $A$ has a bounded resolution of finitely
			generated projective $A$-bimodules,
			\item  There is an isomorphism
			$$\Ext_{A^e}^i(A,A^e)\cong\begin{cases}0,& i\neq d,
				\\A_\mu(l),&i=d\end{cases}$$
			of graded $A^e$-modules for some $l\in \ZZ$ and some  $\mu\in \GrAut(A)$.
		\end{enumerate}
		If such an automorphism $\mu$ exists, it is unique up to an inner automorphism and is called the \emph{Nakayama automorphism} of $A$. The integer  $l$ is also called the \emph{AS-index} of $A$. A \emph{CY} algebra is a twisted CY algebra whose Nakayama automorphism is an inner automorphism.
		
	\end{defn}
	
	\begin{rem}
		An AS-regular algebra must be twisted CY. Actually, in the definition of an AS-regular algebra, the condition that requires finite GK-dimension is sometimes omitted (e.g. \cite{ms,rrz}). If this condition is omitted, then a connected graded algebra is twisted CY of dimension $d$ if and only if it is AS-regular of dimension $d$. In that case, the AS-indexes are the same  \cite[Lemma 1.2]{rrz}. The \emph{Nakayama automorphism} of an AS-regular algebra is defined as its Nakayama automorphism as a twisted CY algebra.
	\end{rem}


	\section{The distributive Yang-Baxter algebras}
	The aim of this section is to determine the algebraic structure of the
	YB algebra associated with a distributive solution.
	\subsection{Distributive solutions}\begin{defn}  
		A solution $(X, r)$ of the YBE is called \emph{distributive} if for all $x, y, z \in X$,
		\[
		{}^{x}({}^{y}z) = {}^{{}^{x}y}({}^{x}z).
		\]
		Equivalently, the left actions satisfy
		\[
		\mathcal{L}_x \mathcal{L}_y = \mathcal{L}_{{}^{x}y} \mathcal{L}_x	\]  for all $ x, y \in X$.
		
		A YB algebra is called \emph{distributive}, if it is obtained by a distributive solution.
	\end{defn}
	
	It is well known that there is a one-to-one correspondence between solutions of the YBE and involutive biracks, algebraic structures equipped with two binary operations, each forming a (one-sided) quasigroup, that satisfy certain compatibility identities (see e.g. \cite[Lemma 1.2]{de}). Consequently, by \cite[Corollary 5.7]{jpz}, distributive solutions admit an equivalent characterization in terms of right actions.
	
	\begin{lem}
		Let $(X,r)$ be a solution. The following conditions are equivalent:
		\begin{enumerate}
			\item (X,r) is distributive;
			\item For every $x,y,z\in X$:
			\[(z^y)^x=(z^x)^{(y^x)} \quad (\text{equivalently},  \mc{R}_x\mc{R}_y=\mc{R}_{y^x}\mc{R}_x).\]
		\end{enumerate}
	\end{lem}

	\begin{lem}\label{lem action}
		Let $(X,r)$ be a square-free solution. Then the followings are equivalent:
		\begin{enumerate}
			\item The solution $(X,r)$ is distributive;
			\item For each $x\in X$, $\mc{L}_x$ is an automorphism of the solution $(X,r)$;
			\item The permutation group $\mc{G}(X,r)$ acts naturally on the YB algebra $\mc{A}(X,r)$.
		\end{enumerate}	
	\end{lem}
	\proof  The solution $(X,r)$ is square-free. Therefore,  it satisfies the Condition {\bf lri}. By Remark \ref{rem auto}, for each $x\in X$, $\mc{L}_x$ is an automorphism of the solution $(X,r)$ if and only if $\mc{L}_x \mc{L}_y=\mc{L}_{\mc{L}_x(y)}\mc{L}_x$ for any $y\in X$. Consequently, (1) and (2) are equivalent.

	The  permutation group $\mc{G}(X, r)$ is generated by $\{\mc{L}_x\mid x\in X\}$. The equivalence of (2) and (3) follows from Lemma \ref{lem auto}.  \qed

	\subsection{Trivial affine meshes}
	Every distributive solution $(X, r)$ can be constructed by a trivial affine mesh.
	\begin{defn}
		A \emph{trivial affine mesh} over a nonempty set $I$ is a pair
		\[
		\mathcal{T}=\left((G_i)_{i\in I},(d_{i,j})_{i,j\in I}\right),
		\]
		where each $(G_i,+)$ is an abelian group and the constants $d_{i,j}\in G_j$ satisfy
		\[
		G_j=\langle d_{i,j}\mid i\in I\rangle, 		
		\]
	for every $j\in I$. If $I$ is a finite set, then a trivial affine mesh is displayed as a pair $\left(\left(G_{i}\right)_{i \in I},D\right)$, where $D=(d_{i,j})_{i,j\in I}$ is a $|I|\times |I|$ matrix.
		
		Referring to \cite{jpz}, we call $\mathcal{T}$ \emph{idempotent} if $d_{i,i}=0$ for all $i\in I$.
	\end{defn}
	The following lemma is \cite[Theorem 7.8]{jpz}. Note that we only consider finite solutions.
	
	\begin{lem}\label{L T}
		Every finite distributive solution $(X, r)$ is determined by a trivial affine mesh. Specifically, the set $X$ decomposes as a disjoint union of finite abelian groups
		$X=\bigcup_{j=1}^pG_{j},$ where each $G_j$ is generated by the elements $\{d_{i,j} \mid 1 \leqslant i\leqslant p\}$, i.e.,
		$G_j = \left\langle d_{i,j} \mid 1 \leqslant i \leqslant p \right\rangle$. 	
		The map $r : X \times X \to X \times X$ is then given by
		$$
		r(x, y) = (y + d_{i,j},\; x - d_{j,i} ),
		$$
		for all $x \in G_i$ and $y \in G_j$.
	\end{lem}
	
	\begin{rem}
		Let $\mathcal{T}=\left(\left(G_{i}\right)_{1\le i\le p}, (d_{i,j})_{1\le i,j\le p}\right)$ be a trivial affine mesh with finite groups $G_i$. For the convenience of discussing the associated YB algebra, in the following, we label the elements of the corresponding solution $(X,r)$ by
		\begin{equation}\label{eq s1}
			X=\{x_{i,a}\mid 1\le i\le p,\ a\in G_i\}.
		\end{equation}
		Then
		\begin{equation}\label{eq s2}
			r(x_{i,a},x_{j,b})=(x_{j,b+d_{i,j}},\;x_{i,a-d_{j,i}})
		\end{equation}
		for all $1\le i,j\le p$, $a\in G_i$, and $b\in G_j$. The corresponding  YB  algebra is the algebra  $\mc{A} = \kk\langle X\rangle/(\mathbf{R})$, where
		\[
		\mathbf{R} = \{x_{i,a}x_{j,b} - x_{j,b+d_{i,j}}x_{i,a-d_{j,i}} \mid 1 \le i, j \le p, a \in G_i, b \in G_j\}.
		\]
		
		It is easy to see that $(X,r)$ is square-free if and only if $\mathcal{T}$ is idempotent. Recall that every finite  square-free solution gives rise to an AS-regular algebra. Every  idempotent trivial affine mesh with finite groups determines an AS-regular algebra.

	\end{rem}
	
	\begin{notation}\label{notation}
		Throughout the remainder of the paper, we use the following
		notation for a trivial affine mesh
		$\mathcal{T}=((G_i)_{1\le i\le p},(d_{i,j})_{1\le i,j\le p})$
		with finite groups $G_i$.
		For each $1\le i\le p$, fix an isomorphism of abelian groups
		\[
		G_i\cong
		\ZZ_{m_{i,1}}\times\ZZ_{m_{i,2}}\times\cdots
		\times\ZZ_{m_{i,s_i}},
		\]
		where $\ZZ_m=\ZZ/m\ZZ$ is regarded as an additive group,
		$s_i\ge1$, and $m_{i,1},\cdots,m_{i,s_i}$ are positive
		integers satisfying $m_{i,k}\mid m_{i,k+1}$ for
		$1\le k<s_i$.
		We identify $G_i$ with this direct product via the chosen
		isomorphism. If $G_i$ is trivial, we take $s_i=1$ and
		$m_{i,1}=1$.
		
		For an integer $c$, we write $\bar c$ for its residue class, with the modulus being clear from the context. Thus every $a\in G_i$ has a unique expression
		\[
		a=(\bar a_1,\cdots,\bar a_{s_i}),
		\qquad 0\le a_k<m_{i,k},\quad1\le k\le s_i,
		\]
		where $a_k$ is the least nonnegative integer representative of the $k$-th coordinate. Group addition and subtraction are performed coordinatewise in the corresponding residue class groups. 		 We use the following notation.
		\begin{itemize}
			\item
			Set $N_i=|G_i|=\prod_{k=1}^{s_i}m_{i,k}$.
			
			\item
			For $1\le j\le p$, write
			\[
			d_{j,i}=(\overline{d_{j,i,1}},\cdots,\overline{d_{j,i,s_i}}),
			\qquad 0\le d_{j,i,k}<m_{i,k},\quad1\le k\le s_i.
			\]
			
			\item
			When multiplication is needed, we regard each $\ZZ_{m_{i,k}}$ as the residue class ring $\ZZ/m_{i,k}\ZZ$. For $a=(\bar a_1,\cdots,\bar a_{s_i})$ and $b=(\bar b_1,\cdots,\bar b_{s_i})$ in $G_i$, define
			\[
			a\cdot b=(\overline{a_1b_1},\cdots,\overline{a_{s_i}b_{s_i}}).
			\]
			This coordinatewise multiplication depends on the chosen identification of $G_i$ with the above direct product.
			
			\item
			For each $1\le k\le s_i$, choose a primitive $m_{i,k}$-th root of unity $\xi_{i,k}\in\kk$, and set $\xi_{G_i}=(\xi_{i,1},\cdots,\xi_{i,s_i})$. For $a=(\bar a_1,\cdots,\bar a_{s_i})\in G_i$, define
			\[
			\xi_{G_i}^{a}=\prod_{k=1}^{s_i}\xi_{i,k}^{a_k}.
			\]
			This expression is independent of the choice of integer representatives of the coordinates. In particular,
			\[
			\xi_{G_i}^{-a}=(\xi_{G_i}^{a})^{-1},
			\qquad
			\xi_{G_i}^{a\cdot b}=\prod_{k=1}^{s_i}\xi_{i,k}^{a_kb_k}.
			\]
			
			\item
			We order the elements of $G_i$ lexicographically by their least nonnegative coordinate representatives, reading the coordinates from right to left. Explicitly, for $a=(\bar a_1,\cdots,\bar a_{s_i})$ and $b=(\bar b_1,\cdots,\bar b_{s_i})$ in $G_i$, with $0\le a_k,b_k<m_{i,k}$, we define $a<b$ if there exists $k\in\{1,\cdots,s_i\}$ such that
			\begin{equation}\label{order}
				a_\ell=b_\ell\quad\text{for all }k<\ell\le s_i, \text{ and }
				a_k<b_k.
			\end{equation}

		\end{itemize}
	\end{notation}
	
	The following lemma will be needed later.
	\begin{lem}
		Let
		$\mathcal{T}=((G_i)_{1\le i\le p},(d_{i,j})_{1\le i,j\le p})$
		be a trivial affine mesh with finite groups $G_i$, and set
		\[
		\widehat N_{i,k}=\frac{N_i}{m_{i,k}},\]
		for \(1\le i\le p,\ 1\le k\le s_i\).
				Then, for all $1\le i,j\le p$,
		\begin{equation}\label{eq xi}
			\prod_{a\in G_i}\xi_{G_i}^{-a\cdot d_{j,i}}
			=
			(-1)^{\sum_{k=1}^{s_i}
				(N_i-\widehat N_{i,k})d_{j,i,k}}.
		\end{equation}
	\end{lem}
	
	\proof It is clear that if $\xi$  is a primitive $m$-th root of unity, then $$\prod_{t=0}^{m-1} \xi^t = (-1)^{m-1}.$$
	Consequently, \begin{align*}
		\prod_{a \in G_i} \xi_{G_i}^{-a\cdot d_{j,i}} & =\prod_{a_1=0}^{m_{i,1}-1}\cdots
		\prod_{a_{s_i}=0}^{m_{i,s_i}-1}
		\left(\prod_{k=1}^{s_i}\xi_{i,k}^{-a_kd_{j,i,k}}\right)\\
		& =\prod_{k=1}^{s_i}
		\left(
		\prod_{a_k=0}^{m_{i,k}-1}
		\xi_{i,k}^{-a_kd_{j,i,k}}
		\right)^{N_i/m_{i,k}}\\
		&=(-1)^{\sum_{k=1}^{s_i}
			\widehat N_{i,k}(m_{i,k}-1)d_{j,i,k}}\\
		&=(-1)^{\sum_{k=1}^{s_i}
			(N_i-\widehat N_{i,k})d_{j,i,k}}.
	\end{align*} \qed

	\subsection{Skew polynomial algebras} We will show that the AS-regular algebra obtained by an idempotent trivial affine mesh is isomorphic to a skew polynomial algebra. It can also be derived from \cite{gim2}. Our proof here is more elementary. We need the following lemma first.
	
	\begin{lem}\label{lemma 1}
		Let $m',m,n$ be positive integers with $m'\ge m$,
		and let $\xi\in\kk$ be a primitive $m'$-th root of unity.
		Suppose that $T_0,\ldots,T_{m-1}$ are invertible
		$n\times n$ matrices over $\kk$. Then the block matrix
		\[
		\mc{D}_m=
		\begin{pmatrix}
			T_0 & T_1 & T_2 & \cdots & T_{m-1}\\
			T_0 & \xi T_1 & \xi^2T_2 & \cdots & \xi^{m-1}T_{m-1}\\
			T_0 & \xi^2T_1 & \xi^4T_2 & \cdots & \xi^{2(m-1)}T_{m-1}\\
			\vdots & \vdots & \vdots &  & \vdots\\
			T_0 & \xi^{m-1}T_1 & \xi^{2(m-1)}T_2
			& \cdots & \xi^{(m-1)^2}T_{m-1}
		\end{pmatrix}
		\]
		is invertible. 
	\end{lem}
	
	\proof
	If $m=1$, then $\mc{D}_1=T_0$, so the assertion is immediate.
	Assume that $m\ge2$, and consider the Vandermonde matrix
	\[
	W
	=
	\begin{pmatrix}
		1 & 1 & 1 & \cdots & 1\\
		1 & \xi & \xi^2 & \cdots & \xi^{m-1}\\
		1 & \xi^2 & \xi^4 & \cdots & \xi^{2(m-1)}\\
		\vdots & \vdots & \vdots &   & \vdots\\
		1 & \xi^{m-1} & \xi^{2(m-1)}
		& \cdots & \xi^{(m-1)^2}
	\end{pmatrix}.
	\]
	Since $\xi$ has order $m'\ge m$, the scalars
	$1,\xi,\ldots,\xi^{m-1}$ are pairwise distinct. Thus
	\[
	\det W=\prod_{0\le r<s\le m-1}(\xi^s-\xi^r)\ne0,
	\]
	so $W$ is invertible. Write
	$W^{-1}=(c_{rs})_{0\le r,s\le m-1}$.
	
	Now define
	\[
	B=
	\begin{pmatrix}
		I_n & I_n & I_n & \cdots & I_n\\
		I_n & \xi I_n & \xi^2I_n & \cdots & \xi^{m-1}I_n\\
		I_n & \xi^2I_n & \xi^4I_n & \cdots & \xi^{2(m-1)}I_n\\
		\vdots & \vdots & \vdots & \ddots & \vdots\\
		I_n & \xi^{m-1}I_n & \xi^{2(m-1)}I_n
		& \cdots & \xi^{(m-1)^2}I_n
	\end{pmatrix},\]
	and \(C=(c_{rs}I_n)_{0\le r,s\le m-1}\).
	Since $WW^{-1}=W^{-1}W=I_m$, it is direct to check that 
	\[
	(BC)_{r,s}
	=
	\left(\sum_{t=0}^{m-1}\xi^{rt}c_{ts}\right)I_n
	=
	\delta_{rs}I_n
	\]
	for all $0\le r,s\le m-1$, and similarly $CB=I_{mn}$.
	Hence $B$ is invertible with $B^{-1}=C$.
	
	Finally, we see that
	\[
	\mc{D}_m
	=
	B
	\begin{pmatrix}
		T_0 & 0 & \cdots & 0\\
		0 & T_1 & \cdots & 0\\
		\vdots & \vdots &  & \vdots\\
		0 & 0 & \cdots & T_{m-1}
	\end{pmatrix}.
	\]
	The block diagonal matrix is invertible because each
	$T_s$ is invertible. Therefore $\mc{D}_m$ is a product
	of two invertible matrices and is itself invertible.
	\qed

	\begin{thm}\label{thm}
		Let $\mc{T}=((G_i)_{1\le i\le p},D)$ be an idempotent
		trivial affine mesh with finite groups $G_i$. It determines a square-free solution $(X,r)$ as (\ref{eq s1}) and (\ref{eq s2}).
		Let $\mc{A}=\mc{A}(X,r)$ be the corresponding YB algebra
		and $\mc{G}=\mc{G}(X,r)$ its permutation group.
		For $1\le i\le p$ and $a\in G_i$, set
		\[
		y_{i,a}=\sum_{b\in G_i}\xi_{G_i}^{a\cdot b}x_{i,b}.
		\]
		Then 
		\begin{enumerate}
			\item
			The set
			\[
			Y=\{y_{i,a}\mid 1\le i\le p,\ a\in G_i\}
			\]
			is a basis of $\mc{A}_1$, and $\mc{A}$ is generated by
			$Y$ with defining relations
			\[
			y_{j,b}y_{i,a}
			=
			\xi_{G_i}^{-a\cdot d_{j,i}}
			\xi_{G_j}^{b\cdot d_{i,j}}
			y_{i,a}y_{j,b}
			\]
			for all $1\le i,j\le p$, $a\in G_i$, and $b\in G_j$.
			In particular, $\mc{A}$ is isomorphic to a skew
			polynomial algebra.
			
			\item
			The natural action of $\mc{G}$ on $X$ extends to an
			action on $\mc{A}$ by graded algebra automorphisms.
			This action is diagonal  with respect
			to the basis $Y$, with
			\[
			\mc{L}_{x_{i,a}}(y_{j,b})
			=
			\xi_{G_j}^{-b\cdot d_{i,j}}y_{j,b}
			\]
			for all $1\le i,j\le p$, $a\in G_i$, and $b\in G_j$.
		\end{enumerate}
	\end{thm}
	\proof (1) For each $1\le i\le p$, define the matrices
	$T_i^{(1)},\ldots,T_i^{(s_i)}$ recursively as follows.
	\[
	T_i^{(1)} = \begin{pmatrix}
		1 & 1 & 1 & \cdots & 1 \\
		1 & \xi_{i,1} & \xi_{i,1}^2 & \cdots & \xi_{i,1}^{m_{i,1}-1} \\
		1 & \xi_{i,1}^2 & \xi_{i,1}^4 & \cdots & \xi_{i,1}^{2m_{i,1}-2} \\
		\vdots & \vdots & \vdots & & \vdots \\
		1 & \xi_{i,1}^{m_{i,1}-1} & \xi_{i,1}^{2m_{i,1}-2} & \cdots & \xi_{i,1}^{(m_{i,1}-1)^2}
	\end{pmatrix},
	\]	
	and for \(2 \le k \le s_i\),
	\[
	T_i^{(k)} = \begin{pmatrix}
		T_i^{(k-1)} & T_i^{(k-1)} & T_i^{(k-1)} & \cdots & T_i^{(k-1)} \\
		T_i^{(k-1)} & \xi_{i,k} T_i^{(k-1)} & \xi_{i,k}^2 T_i^{(k-1)} & \cdots & \xi_{i,k}^{m_{i,k}-1} T_i^{(k-1)} \\
		T_i^{(k-1)} & \xi_{i,k}^2 T_i^{(k-1)} & \xi_{i,k}^4 T_i^{(k-1)} & \cdots & \xi_{i,k}^{2m_{i,k}-2} T_i^{(k-1)} \\
		\vdots & \vdots & \vdots & & \vdots \\
		T_i^{(k-1)} & \xi_{i,k}^{m_{i,k}-1} T_i^{(k-1)} & \xi_{i,k}^{2m_{i,k}-2} T_i^{(k-1)} & \cdots & \xi_{i,k}^{(m_{i,k}-1)^2} T_i^{(k-1)}
	\end{pmatrix}.\]
	
	The $(r,s)$-entry of $T_i^{(1)}$ is
	$\xi_{i,1}^{rs}$ for $0\le r,s<m_{i,1}$, and the
	$(r,s)$-block of $T_i^{(k)}$ is
	$\xi_{i,k}^{rs}T_i^{(k-1)}$ for $0\le r,s<m_{i,k}$.
	In particular, if $m_{i,1}=1$, then $T_i^{(1)}=(1)$;
	if $m_{i,k}=1$ for $k\ge2$, then
	$T_i^{(k)}=T_i^{(k-1)}$.
	
	Since $\xi_{i,1}$ is a primitive $m_{i,1}$-th root
	of unity, the scalars
	$1,\xi_{i,1},\ldots,\xi_{i,1}^{m_{i,1}-1}$
	are pairwise distinct. Hence $T_i^{(1)}$ is an
	invertible $m_{i,1}\times m_{i,1}$ Vandermonde matrix.  
	Apply Lemma \ref{lemma 1} inductively, we obtain that  $T_i^{(s_i)}$ is an $N_i\times N_i$ invertible matrix. 
	
	We index the rows and columns of $T_i^{(s_i)}$ by the
	elements of $G_i$, listed in increasing order with respect
	to~\eqref{order}.
	For $a=(\bar a_1,\cdots,\bar a_{s_i})$ and
	$b=(\bar b_1,\cdots,\bar b_{s_i})$ in $G_i$, the
	$(a,b)$-entry of $T_i^{(s_i)}$ is
	\[
	\bigl(T_i^{(s_i)}\bigr)_{a,b}
	=\xi_{G_i}^{a\cdot b}
	=\prod_{k=1}^{s_i}\xi_{i,k}^{a_kb_k}.
	\]
	
	Define the row vectors
	\[
	X_i=(x_{i,a})_{a\in G_i},
	\qquad
	Y_i=(y_{i,a})_{a\in G_i},
	\]
	where the entries are arranged in the same order.
	Since
	\[
	y_{i,a}=\sum_{b\in G_i}\xi_{G_i}^{a\cdot b}x_{i,b}
	\]
	and $\xi_{G_i}^{a\cdot b}=\xi_{G_i}^{b\cdot a}$,
	we have $Y_i=X_iT_i^{(s_i)}$. Consequently,
	\[
	(Y_1,\ldots,Y_p)
	=
	(X_1,\ldots,X_p)
	\begin{pmatrix}
		T_1^{(s_1)} & & & \\
		& T_2^{(s_2)} & & \\
		& & \ddots & \\
		& & & T_p^{(s_p)}
	\end{pmatrix}.
	\]
	
	Each diagonal block $T_i^{(s_i)}$ is invertible, then this change-of-basis matrix is invertible. 
	As $X$ is a basis of $\mc{A}_1$, it follows that
	$Y$ is also a basis of $\mc{A}_1$ and generates $\mc{A}$. They are subject to the following relations
	\begin{equation}\label{eq relation}
		y_{j,b}y_{i,a}=\xi_{G_i}^{-a\cdot d_{j,i}}\xi_{G_j}^{b\cdot d_{i,j}}y_{i,a}y_{j,b}.
	\end{equation}
	In fact,
	\begin{align*}
		y_{j,b}y_{i,a}&=\sum_{c'\in G_i,c''\in G_j}\xi_{G_i}^{a\cdot c'}\xi_{G_j}^{b\cdot c''}x_{j,c''}x_{i,c'}
		\\
		&=\sum_{c'\in G_i,c''\in G_j}\xi_{G_i}^{a\cdot c'}\xi_{G_j}^{b\cdot c''}x_{i,c'+d_{j,i}}x_{j,c''-d_{i,j}}\\
		&=\xi_{G_i}^{-a\cdot d_{j,i}}\xi_{G_j}^{b\cdot d_{i,j}}\sum_{c'\in G_i,c''\in G_j}\xi_{G_i}^{a\cdot c'}\xi_{G_j}^{b\cdot c''}x_{i,c'}x_{j,c''}\\
		&=\xi_{G_i}^{-a\cdot d_{j,i}}\xi_{G_j}^{b\cdot d_{i,j}}y_{i,a}y_{j,b}.
	\end{align*}
	It remains to show that these are defining relations. 	Set $N=|X|=\sum_{i=1}^pN_i$, and write
	$\mc{A}=T(V)/(\bf{R})$, where $V=\mc{A}_1$ and
	${\bf R}\subseteq V\otimes V$ is the space of defining
	quadratic relations.
	By Lemma~\ref{lem skew}  we obtain that $\dim_{\kk}${\bf{R}}$=\binom{N}{2}$.
	
	With  a total order on $Y$, the  relations (\ref{eq relation})  give $\binom{N}{2}$ linearly independent
	elements of $\bf{R}$. Hence they form a basis of $\bf{R}$.
	Therefore these relations present $\mc{A}$ as a skew
	polynomial algebra.

	(2) By Lemma~\ref{lem action}, the natural action of
	$\mc{G}$ on $X$ extends uniquely to an action on
	$\mc{A}$ by graded algebra automorphisms. We use
	the same notation for a permutation and its extension.  The permutation group $\mc{G}$ is generated by $\{\mc{L}_{x_{i,a}}\mid 1\le i\le p, a\in G_i\}$. For $1\le i,j\le p$, $a\in G_i$, and $b\in G_j$,
	\begin{align*}
		\mathcal{L}_{x_{i, a}}(y_{j,b})& = \mathcal{L}_{x_{i, a}} (\sum_{c\in G_j} \xi^{b\cdot c}_{G_j}  x_{j,c})\\
		& = \sum_{c\in G_j} \xi^{b\cdot c}_{G_j}  x_{{j,c} + d_{i,j}}\\
		& =  \xi_{G_j}^{-b\cdot  d_{i,j}}\sum_{c\in G_j} \xi^{b\cdot (c + d_{i,j})}_{G_j}  x_{j,c + d_{i,j}}\\
		& =  \xi_{G_j}^{-b\cdot  d_{i,j}}y_{j,b}.
	\end{align*}
	This completes the proof of (2).
	
	\qed

	\subsection{The Calabi-Yau property} 	With Theorem \ref{thm}, we are able to describe the Nakayama automorphism of distributive YB algebras.

	\begin{thm}
		With the notation and assumptions of Theorem~\ref{thm},
		the algebra $\mc{A}$ is a Koszul twisted Calabi-Yau
		algebra of dimension $\sum_{i=1}^pN_i$.
		Its Nakayama automorphism $\mu$ is given by
		\[
		\mu(x_{i,a})
		=
		\varepsilon_i\,x_{i,\,a-t_i}
		\]
		for all $1\le i\le p$ and $a\in G_i$, where
		\[
		\varepsilon_i
		=
		(-1)^{\sum_{j=1}^p\sum_{k=1}^{s_j}
			(N_j-\widehat N_{j,k})d_{i,j,k}}
		\in\{\pm1\},
		\qquad
		t_i=\sum_{j=1}^pN_jd_{j,i}\in G_i.
		\]
		Here $N_jd_{j,i}$ denotes the integer multiple of
		$d_{j,i}$ in the additive group $G_i$.
	\end{thm}
	
	\proof
	It is well known that a skew polynomial algebra on $N$
	generators is a Koszul twisted Calabi-Yau algebra of
	dimension $N$, and its Nakayama automorphism is explicitly
	known; see, for example, \cite[Example~5.5]{rrz} for the
	formula.
	
	Hence, by Theorem~\ref{thm}, $\mc{A}$ is a Koszul twisted Calabi-Yau algebra
	of dimension $\sum_{i=1}^pN_i$. Applying the formula for the Nakayama
	automorphism, we obtain
	\begin{align*}
		\mu(y_{i,a})&=\prod_{(j,b)\neq (i,a)} \xi_{G_i}^{a\cdot d_{j,i}}\xi_{G_j}^{-b\cdot d_{i,j}}y_{i,a}\\
		&=
		\left(
		\prod_{j=1}^p\prod_{b\in G_j}
		\xi_{G_i}^{a\cdot d_{j,i}}
		\xi_{G_j}^{-b\cdot d_{i,j}}
		\right)y_{i,a}\\
		&=
		\left(
		\prod_{j=1}^p
		\bigl(\xi_{G_i}^{a\cdot d_{j,i}}\bigr)^{N_j}
		\prod_{b\in G_j}\xi_{G_j}^{-b\cdot d_{i,j}}
		\right)y_{i,a}\\	
		&\overset{(\ref{eq xi})}{=}\left(\prod_{j=1}^{p}(\xi_{G_i}^{a\cdot d_{j,i}}) ^{N_j}(-1)^{\sum\limits_{k=1}^{s_j}(N_j-\widehat{N}_{j,k})d_{i,j,k}}\right)y_{i,a}\\
		&=\varepsilon_i
		\xi_{G_i}^{a\cdot t_i}y_{i,a},
	\end{align*}
	where the second equality holds because the factor
	indexed by $(j,b)=(i,a)$ equals $1$, as $d_{i,i}=0$.

	To express $\mu$ in terms of the generators $x_{i,a}$,
	we use the character orthogonality relation
	\[
	\sum_{b\in G_i}\xi_{G_i}^{b\cdot c}
	=
	\begin{cases}
		N_i,&c=0,\\
		0,&c\ne0,
	\end{cases}
	\]
for any \(c\in G_i\).	Indeed, writing $c=(\bar c_1,\ldots,\bar c_{s_i})$, we have
	\begin{align*}
		\sum_{b\in G_i}\xi_{G_i}^{b\cdot c}
		&=
		\sum_{b_1=0}^{m_{i,1}-1}\cdots
		\sum_{b_{s_i}=0}^{m_{i,s_i}-1}
		\prod_{k=1}^{s_i}\xi_{i,k}^{b_kc_k}\\
		&=
		\prod_{k=1}^{s_i}
		\left(
		\sum_{u=0}^{m_{i,k}-1}\xi_{i,k}^{u c_k}
		\right).
	\end{align*}
	Since each $\xi_{i,k}$ is a primitive $m_{i,k}$-th root
	of unity, the $k$-th factor equals $m_{i,k}$ if
	$\bar c_k=0$ and vanishes otherwise.
	
	The definition of $y_{i,b}$ gives
	\begin{align*}
		\frac{1}{N_i}\sum_{b\in G_i}
		\xi_{G_i}^{-a\cdot b}y_{i,b}
		&=\frac{1}{N_i}\sum_{b\in G_i}
		\xi_{G_i}^{-a\cdot b}(\sum_{c\in G_i}\xi_{G_i}^{b\cdot c}x_{i,c})
		\\
		&=\frac{1}{N_i}\sum_{c\in G_i}
		\left(\sum_{b\in G_i}
		\xi_{G_i}^{b\cdot(c-a)}\right)x_{i,c}\\
		&=x_{i,a}.
	\end{align*}

	Consequently,
	\begin{align*}
		\mu(x_{i,a})
		&=
		\frac{1}{N_i}
		\sum_{b\in G_i}
		\xi_{G_i}^{-a\cdot b}\mu(y_{i,b})\\
		&=
		\frac{\varepsilon_i}{N_i}
		\sum_{b\in G_i}
		\xi_{G_i}^{-(a-t_i)\cdot b}y_{i,b}\\
		&=
		\varepsilon_i x_{i,a-t_i},
	\end{align*}
	as required.
	\qed

	\section{The Yang-Baxter permutation group actions}
	
	Throughout this section, we retain the notation of
	Section~2, in particular Notation~\ref{notation}
	and Theorem~\ref{thm}.
	Let $\mathcal{T}=((G_i)_{1\le i\le p},D)$ be an idempotent
	trivial affine mesh with finite groups $G_i$, and let
	$(X,r)$ be the associated square-free distributive solution,
	where
	\[
	X=\{x_{i,a}\mid 1\le i\le p,\ a\in G_i\}.
	\]
	Set $\mc A=\mc A(X,r)$ and
	\[
	Y=\{y_{i,a}\mid 1\le i\le p,\ a\in G_i\},
	\qquad
	y_{i,a}=\sum_{b\in G_i}\xi_{G_i}^{a\cdot b}x_{i,b}.
	\]
	By Theorem~\ref{thm}, $Y$ is a basis of $\mc A_1$, and
	$\mc A$ is a  skew polynomial algebra.
	
	Recall that the permutation group $\mc{G}=\mc{G}(X,r)$
	is the subgroup of $\Sym(X)$ generated by the permutations
	$\mc{L}_{x_{i,a}}$. Its action on $X$ is given by
	\[
	\mc{L}_{x_{i,a}}(x_{j,b})=x_{j,b+d_{i,j}}
	\]
	for all $1\le i,j\le p$, $a\in G_i$, and $b\in G_j$.
	This formula is independent of $a\in G_i$. Throughout this section, write
	\[
	g_i=\mc L_{x_{i,0}},
	\]
	for $1\le i\le p$, so that
	\[
	\mc G=\langle g_1,\ldots,g_p\rangle.
	\]
	
	By Lemma~\ref{lem action}, the natural action of $\mc{G}$
	on $X$ extends uniquely to an action on $\mc{A}$ by
	graded algebra automorphisms. This action is faithful.
	Indeed, if an element $g\in\mc{G}$ acts as the identity
	on $\mc{A}$, then it fixes every element of the basis
	$X$ of $\mc{A}_1$. Since $\mc{G}\leqslant\Sym(X)$, this implies
	$g=\mathrm{id}_X$. We therefore identify
	$\mc G$ with its image in $\GrAut(\mc A)$, using the same
	notation for a permutation of $X$ and its extension to $\mc A$.

	It is useful to describe $\mc G$ as a group of translations.  Set
	\[
	P=\prod_{j=1}^pG_j,\qquad
	d_i=(d_{i,1},\ldots,d_{i,p})\in P,\qquad
	H=\langle d_1,\ldots,d_p\rangle\le P.
	\]
	For $u=(u_1,\ldots,u_p)\in P$, define a permutation
	$\tau_u\in\operatorname{Sym}(X)$ by
	\[
	\tau_u(x_{j,b})=x_{j,b+u_j},
	\qquad1\le j\le p,\ b\in G_j.
	\]
	The map $u\mapsto\tau_u$ is an injective group
	homomorphism from $P$ to $\Sym(X)$.
	Since $\tau_{d_i}=g_i$, its restriction to $H$
	is an isomorphism onto $\mc G$. Denote its inverse by
	\begin{equation}\label{eq translation group}
		\begin{aligned}
			\trv:\mc G&\xrightarrow{\sim}H,\\
			g&\longmapsto \trv(g)=(\trv_1(g),\cdots,\trv_p(g)).
		\end{aligned}
	\end{equation}
	Thus
	\[
	g(x_{j,b})=x_{j,b+\trv_j(g)},
	\qquad
	\trv(g_i)=d_i,
	\qquad
	\trv_j(g_i)=d_{i,j}.
	\]
	In particular, $\mc G$ is finite abelian, and
	\[
	\trv(gh)=\trv(g)+\trv(h),
	\]
	for all $g,h\in\mc G$. 
	The mesh condition implies that the projection of $H$
	onto each $G_j$ is surjective, since
	\(
	G_j=\langle d_{1,j},\ldots,d_{p,j}\rangle.
	\)
	
	By Theorem~\ref{thm},
	\[	g_i(y_{j,b})
	=\xi_{G_j}^{-b\cdot d_{i,j}}y_{j,b}.
	\]
	More generally, for every $g\in\mc G$, $1\le j\le p$ and $b\in G_j$,
	\begin{equation}\label{eq translation eigenvalues}
		g(y_{j,b})
		=\xi_{G_j}^{-b\cdot \trv_j(g)}y_{j,b}.
	\end{equation}
	Thus $\mc G$ acts diagonally on $\mc A_1$ with respect
	to the basis $Y$.

	\subsection{Homological determinants}The notion of the homological determinant was introduced by
	J{\o}rgensen and Zhang~\cite[Definition~2.3]{joz} as a noncommutative
	analogue of the usual determinant of a linear group action.
	
	Let $A$ be an AS-Gorenstein algebra of injective dimension $d$.
	Then the local cohomology groups (see, e.g., \cite{joz} for the definition) of $A$ satisfy
	\[
	H_{\mathfrak{m}}^{i}(A) \cong
	\begin{cases}
		0, & i \neq d,\\
		{}_{A}A^{\vee}(l), & i = d,
	\end{cases}
	\]
	where $A^{\vee} = \bigoplus_{i \in \mathbb{Z}} {\Hom}_k(A_{-i},\kk)
	$
	is the graded dual of $A$.
	
	Each graded automorphism $\sigma$ of $A$ induces a $\sigma$-linear homomorphism
	\[
	H_{\mathfrak{m}}^{d}(\sigma): H_{\mathfrak{m}}^{d}(A) \longrightarrow H_{\mathfrak{m}}^{d}(A).
	\]
	This induces a $\sigma$-linear homomorphism on the dual space:
	\[
	H_{\mathfrak{m}}^{d}(\sigma): {}_{A}A^{\vee}(l) \longrightarrow {}_{A}A^{\vee}(l).
	\]
	
	There exists a nonzero scalar $c \in \kk^\times$ such that this $\sigma$-linear homomorphism is equal to
	\[
	c\,(\sigma^{-1})^{\vee}: {}_{A}A^{\vee}(l) \longrightarrow {}_{A}A^{\vee}(l).
	\]
	The \emph{homological determinant} of $\sigma$, denoted by $\hdet(\sigma)$,
	is defined to be the scalar $c^{-1}$.
	This yields a group homomorphism $\hdet:\GrAut(A)\to\kk^{\times}$.
	
	Equivalently, $\hdet$ is characterized by the commutativity of the diagram
	\[
	\begin{array}{ccc}
		H_{\mathfrak{m}}^{d}(A) &
		\xrightarrow{\eqmakebox[M]{$H_{\mathfrak{m}}^{d}(\sigma)$}} &
		H_{\mathfrak{m}}^{d}(A) \\
		\bigg\downarrow{\cong} & & \bigg\downarrow{\cong} \\
		A^{\vee}(l) &
		\xrightarrow{\eqmakebox[M]{$\hdet(\sigma)^{-1}(\sigma^{-1})^{\vee}$}} &
		A^{\vee}(l)
	\end{array}
	\]
	for all $\sigma\in\GrAut(A)$.
	
	If $A = \kk[x_1, \dots, x_n]$, then $\GrAut(A)=\GL_n(\kk)$ and the homological determinant recovers the usual determinant.  We refer to \cite[Section 2]{joz} for the details.

	Let  $A = T(V)/(\mathbf{R})$ be a quadratic algebra. The homogeneous dual  of $A$ is
	defined as  $A^!=T(V^*)/(\mathbf{R}^\perp)$, where $\mathbf{R}^\perp\subseteq (V\otimes V)^*\cong V^*\otimes V^*$ is the orthogonal complement of $\mathbf{R}$. Assume $\dim V = n$ and $\{x_1, x_2, \dots, x_n\}$ is a basis of $V$, with  $\{x_1^*, x_2^*, \dots, x_n^*\}$ being the dual basis of $V^*$.
	Let $\sigma$ be a graded automorphism of $A$.  The dual of the restriction $\sigma|_{A_1}$ induces a graded automorphism $\sigma^\tau$ of $A^!$. Suppose $\sigma$ is defined by
	\[
	\sigma(x_1, x_2, \dots, x_n) = (x_1, x_2, \dots, x_n) M,
	\]
	where $M = (m_{ij})$ is an $n \times n$ invertible matrix. Then $\sigma^\tau$ is determined by
	\[
	\sigma^\tau(x_1^*, x_2^*, \dots, x_n^*) = (x_1^*, x_2^*, \dots, x_n^*) M^T,
	\]
	where $M^T$ is the transpose matrix of $M$.
	
	\begin{rem}
		Note that in this paper we use the reverse convention that
		\[
		\left(f_{n} \otimes f_{n-1} \otimes \cdots \otimes f_{1}\right)
		\left(x_{1} \otimes x_{2} \otimes \cdots \otimes x_{n}\right)
		= f_{1}\left(x_{1}\right) f_{2}\left(x_{2}\right) \cdots f_{n}\left(x_{n}\right)
		\]
		to identify $(V^*)^{\otimes n}$ with $(V^{\otimes n})^*$. 
	\end{rem}
	
	If $A$ is a  Koszul AS-regular algebra of global dimension $d$, then the Yoneda algebra $E(A) = \bigoplus_{i\geqslant 0} \operatorname{Ext}^i_A(\kk, \kk)$ is anti-isomorphic to $A^!$. In particular, we have $(A^!)_d\cong\operatorname{Ext}^d_A(\kk,\kk)$ and is one-dimensional. The following lemma can be derived from \cite[Proposition 1.11]{wz}
	\begin{lem}\label{lem hddet}
		Let $A$ be a Koszul AS-regular algebra of global dimension $d$ and
		$\sigma\in\GrAut(A)$  with $\sigma^\tau$ the graded automorphism of $A^!$ induced by $\sigma$. Then
		$$\hdet(\sigma)w=\sigma^\tau (w)$$
		for any $w\in A^!_d$.
	\end{lem}

	\begin{prop}\label{prop mesh hdet}
		For every $g\in\mc G$,
		\[
		\hdet(g)
		=\prod_{j=1}^p\prod_{b\in G_j}
		\xi_{G_j}^{-b\cdot \trv_j(g)}.
		\]
		In particular, for $1\le i\le p$,
		\[
		\hdet(g_i)
		=(-1)^{\sum_{j=1}^p\sum_{k=1}^{s_j}
			(N_j-\widehat N_{j,k})d_{i,j,k}}.
		\]
		
	\end{prop}	
	\proof The algebra $\mc{A}$ is isomorphic to a skew polynomial algebra of global dimension $N=\sum_{i=1}^pN_i$. Its Koszul dual $\mc{A}^!$ is isomorphic to the algebra generated by
	$\{y^*_{i,a}\mid 1\leqslant i\leqslant p, a\in G_i\}$ subject to the relations:
	\[
	y^*_{j,b}\wedge y_{i,a}^* = -\xi_{G_i}^{-a\cdot d_{j,i}}\xi_{G_j}^{b \cdot d_{i,j}} y_{i,a}^*\wedge y_{j,b}^*,
	\] for any $1\leqslant i,j\leqslant p$ and $a\in G_i, b\in G_j$. Its top component $\mc{A}_N^!$ is spanned by the nonzero
	ordered product
	\[
	w=w_1\wedge w_2 \wedge\cdots \wedge w_p,
	\qquad
	w_j=\bigwedge_{b\in G_j}y^*_{j,b},
	\]
	where the factors in each $w_j$ are arranged in increasing
	order with respect to~\eqref{order}.
	
	For any $g\in \mc{G}$, following from \eqref{eq translation eigenvalues},
	\[
	g^\tau(w)=
	\left(\prod_{j=1}^p\prod_{b\in G_j}
	\xi_{G_j}^{-b\cdot \trv_j(g)}\right)w.
	\]
	The first assertion follows from Lemma~\ref{lem hddet}, and
	the formula for $g_i$ follows from (\ref{eq xi}).
	\qed
	
	\subsection{Quasi-reflections}
	
	Let $A$ be a graded algebra and  $g\in \GrAut(A)$. Then the trace function of  $g$  is defined to be
	$$\Tr_{A}(g, t)=\sum_{i=0}^{\infty}\tr(g|_{A_{i}}) t^{i} \in k[\![ t]\!],$$
	where  $\tr(g|_{A_{i}})$  is the trace of the linear map  $g|_{A_{i}}$. In particular, $\Tr_A(\id_A,t)= \sum_{i=0}^{\infty}\dim A_i t^{i}=H_A(t)$.
	\begin{defn}
		Let $A$ be a noetherian connected graded AS-regular
		algebra of global dimension $N$ with Hilbert series
		\[
		H_A(t)=\frac{1}{(1-t)^N},
		\]
		and let $g\in\GrAut(A)$ have finite order.
		We call $g$ a \emph{quasi-reflection} if
		\[
		\operatorname{Tr}_A(g,t)
		=\frac{1}{(1-t)^{N-1}(1-\lambda t)}
		\]
		for some $\lambda\in\kk^\times\setminus\{1\}$.
		We call $g$ a \emph{reflection} on $A_1$ if
		\[
		\operatorname{rank}
		\bigl(g|_{A_1}-\operatorname{id}_{A_1}\bigr)=1.
		\]
	\end{defn}
	
	We refer to	\cite{kkz1,kkz2} for the details about these notions. 
	
	\begin{lem}\label{lem ref}
		Let $g\in\mc G$, and let $o_j$ be the order of $\trv_j(g)$ in $G_j$,
		with $o_j=1$ when $\trv_j(g)=0$. Then
		\[
		\operatorname{Tr}_{\mc A}(g,t)
		=\prod_{j=1}^p(1-t^{o_j})^{-N_j/o_j}.
		\]
		Moreover, the following conditions are equivalent:
		\begin{enumerate}
			\item $g$ is a quasi-reflection of $\mc A$;
			\item $g$ is a reflection on $\mc A_1$;
			\item there is a unique $j$ such that $\trv_j(g)\ne0$, and $|G_j|=2$.
		\end{enumerate}
		In particular, $g_i$ is a quasi-reflection if and only if
		$d_{i,j}=0$ except at one index $j$ with $|G_j|=2$ and
		$d_{i,j}\ne0$. If $|G_j|\ne2$ for every $j$, then $\mc G$
		contains no reflections or quasi-reflections.
	\end{lem}
	\proof Order the generators $y_{j,b}$ first by $j$ and then
	by the order on $G_j$ defined in~\eqref{order}.
	Since $\mc A$ is a skew polynomial algebra, the
	ordered monomials
	\[
	\prod_{j=1}^p\prod_{b\in G_j}y_{j,b}^{\alpha_{j,b}},
	\qquad \alpha_{j,b}\in\mathbb Z^{\ge0},
	\]
	form a basis of $\mc A$, where the products are taken
	in this order.
	Each such monomial has degree
	$\sum_{j=1}^p\sum_{b\in G_j}\alpha_{j,b}$.
	By~\eqref{eq translation eigenvalues}, it is an
	eigenvector of $g$ with eigenvalue
	\[
	\prod_{j=1}^p\prod_{b\in G_j}
	\left(\xi_{G_j}^{-b\cdot \trv_jj(g)}\right)^{\alpha_{j,b}}.
	\]
	Consequently,
	\begin{align*}
		\operatorname{Tr}_{\mc A}(g,t)
		&=
		\sum_{\alpha}
		\prod_{j=1}^p\prod_{b\in G_j}
		\left(\xi_{G_j}^{-b\cdot \trv_j(g)}t\right)^{\alpha_{j,b}}\\
		&=
		\prod_{j=1}^p\prod_{b\in G_j}
		\left(
		\sum_{r=0}^{\infty}
		\left(\xi_{G_j}^{-b\cdot \trv_j(g)}t\right)^r
		\right)\\
		&=
		\prod_{j=1}^p\prod_{b\in G_j}
		\left(1-\xi_{G_j}^{-b\cdot \trv_j(g)}t\right)^{-1},
	\end{align*}
	where $\alpha$ ranges over all families of non-negative
	integers $(\alpha_{j,b})_{1\le j\le p,\ b\in G_j}$.
	
	For each $j$, set
	\[
	V_j=\operatorname{span}_{\kk}\{x_{j,b}\mid b\in G_j\}
	=\operatorname{span}_{\kk}\{y_{j,b}\mid b\in G_j\},
	\]
	so that $\mc A_1=\bigoplus_{j=1}^p V_j$. Computing in the $y$-basis gives
	\[
	\det\bigl(I-t(g|_{V_j})\bigr)
	=\prod_{b\in G_j}
	\left(1-\xi_{G_j}^{-b\cdot \trv_jj(g)}t\right).
	\]
	On the other hand, in the $x$-basis, $g|_{V_j}$ is
	the permutation induced by the translation
	$b\mapsto b+\trv_jj(g)$. Since $\trv_jj(g)$ has order $o_j$,
	this permutation consists of $N_j/o_j$ disjoint
	cycles of length $o_j$.
	For the permutation matrix $C$ of such a cycle,
	\[
	\det(I-tC)=1-t^{o_j}.
	\]
	Hence,  
	we obtain
	\[
	\prod_{b\in G_j}
	\left(1-\xi_{G_j}^{-b\cdot \trv_jj(g)}t\right)
	=\det\bigl(I-t\,g|_{V_j}\bigr)
	=(1-t^{o_j})^{N_j/o_j}.
	\]
	Therefore, 
	\[
	\operatorname{Tr}_{\mc A}(g,t)
	=\prod_{j=1}^p(1-t^{o_j})^{-N_j/o_j}.
	\]
	The trace formula is proved.
	
	The order of the pole at $t=1$ is $\sum_{j=1}^pN_j/o_j$. It is also the dimension of  $(\mc A_1)^g$. Thus either of conditions (1) and (2)
	implies
	\[\sum^{p}_{j=1} \frac{N_j}{o_j}=N-1.\]
	That is, 
	\[
	\sum_{j=1}^p\left(N_j-\frac{N_j}{o_j}\right)=1.
	\]
	
	For $o_j>1$, the corresponding summand is
	$(N_j/o_j)(o_j-1)$, a positive integer. It equals $1$
	only when $N_j=o_j=2$. Hence the above equality is
	equivalent to condition (3). Under condition (3), $g$ has
	exactly one eigenvalue $-1$ and all other eigenvalues are $1$,
	so it is both a reflection and a quasi-reflection.
	\qed

	\subsection{Reflection Groups}
	We first recall the definitions and results needed below,
	under the following assumptions.
	
	\begin{hypo}\label{hyp0}
		Assume the following:
		\begin{enumerate}
			\item $A$ is a noetherian connected graded AS-regular
			algebra that is a domain;
			\item $G$ is a finite group;
			\item $G$ acts faithfully on $A$ by graded algebra
			automorphisms, so that $A$ is a left $\kk G$-module
			algebra;
			\item $G$ acts on $A$ as a \emph{reflection group}
			in the sense that the fixed subring
			\[
			A^G=\{a\in A\mid g(a)=a\text{ for all }g\in G\}
			\]
			is again AS-regular.
		\end{enumerate}
	\end{hypo}
	
	Write $\widehat G=\operatorname{Hom}(G,\kk^\times)$.
	For $\chi\in\widehat G$, define
	\begin{equation}\label{Achi}
		A_\chi=\{a\in A\mid g(a)=\chi(g)a
		\text{ for all }g\in G\}.
	\end{equation}
	Set $K=(\kk G)^*$.
	Since $G$ is finite, the left $\kk G$-action on $A$
	corresponds to a right $K$-coaction.
	Moreover, $\widehat G$ is naturally identified with
	the group $G(K)$ of grouplike elements of $K$.
	Under these identifications, $A_\chi$ agrees with
	the corresponding component defined by the coaction
	in \cite[(E0.1.1)]{kz}.
	
	The homological determinant of the $G$-action is a
	group homomorphism
	$\mathrm{hdet}:G\rightarrow\kk^\times$,
	and hence an element of $\widehat G$.
	Thus both $A_{\mathrm{hdet}}$ and
	$A_{\mathrm{hdet}^{-1}}$ are defined by~\eqref{Achi}.

	\begin{lem}[{\cite[Theorem~0.2]{kz}}]\label{lem j,a}
		Assume Hypothesis~\ref{hyp0}, and set $R=A^G$.
		\begin{enumerate}
			\item There is a nonzero element $\mathbf j_{A,G}\in A$,
			unique up to a nonzero scalar, such that
			$A_{\mathrm{hdet}^{-1}}$ is a free left and right
			$R$-module of rank one generated by $\mathbf j_{A,G}$.
			\item There is a nonzero element $\mathbf a_{A,G}\in A$,
			unique up to a nonzero scalar, such that
			$A_{\mathrm{hdet}}$ is a free left and right
			$R$-module of rank one generated by $\mathbf a_{A,G}$.
			\item The products $\mathbf j_{A,G}\mathbf a_{A,G}$
			and $\mathbf a_{A,G}\mathbf j_{A,G}$ lie in $R$, and
			\[
			\mathbf j_{A,G}\mathbf a_{A,G}
			=c\,\mathbf a_{A,G}\mathbf j_{A,G}
			\]
			for some $c\in\kk^\times$.
		\end{enumerate}
	\end{lem}

	The above lemma allows us to define the following fundamental concepts \cite{kz}.
	
	\begin{defn}\label{defn j,a}
		Assume Hypothesis \ref{hyp0}.
		\begin{enumerate}
			\item The element $\mathbf{j}_{A,G}$ in Lemma \ref{lem j,a} (1) is called the \textit{Jacobian} of the $G$-action on $A$.
			\item The element $\mathbf{a}_{A,G}$ in Lemma \ref{lem j,a} (2) is called the \textit{reflection arrangement} of the $G$-action on $A$.
			\item The element $\mathbf{j}_{A,G}\mathbf{a}_{A,G}$, or equivalently, $\mathbf{a}_{A,G}\mathbf{j}_{A,G}$, in Lemma \ref{lem j,a} (3) is called the \textit{discriminant} of the $G$-action on $A$, and denoted by $\bm{\delta}_{A,G}$.
		\end{enumerate}
	\end{defn}	
	The above concepts are well defined up to a nonzero scalar in $\kk$.
	
	These definitions are the group-action specializations of the
	corresponding notions for Hopf actions in~\cite{kz}.
	

	\begin{thm}\label{thm mesh reflection}
		The permutation group $\mc G$ is a reflection group for the algebra  $\mc A$ if and only if
		\begin{enumerate}
			\item $|G_j|\le2$ for every $1\le j\le p$;
			\item $H=\prod_{j=1}^pG_j$.
		\end{enumerate}
		Under condition (1), set
		$J=\{j\mid |G_j|=2\}$ and $q=|J|$.
		Identify each $G_j$, $j\in J$, with $\mathbb Z_2$,
		and regard $D$ as a matrix over $\mathbb Z_2$,
		with zero columns indexed by $j\notin J$.
		Then condition (2) is equivalent to
		\[
		\operatorname{rank}_{\mathbb Z_2}D=q.
		\]
		When these conditions hold, 
		\[ \mc G\cong(\mathbb Z_2)^q \] and \[
		\mc A^{\mc G}
		=\kk[y_{1,0},\ldots,y_{p,0},\ y_{j,1}^2\ (j\in J)]
		\]
		is a commutative polynomial algebra. Moreover, up to nonzero
		scalar multiples,
		\begin{align*}
			\mathbf j_{\mc A,\mc G}
			&=\mathbf a_{\mc A,\mc G}
			=\prod_{j\in J}(x_{j,0}-x_{j,1}),\\
			\boldsymbol\delta_{\mc A,\mc G}
			&=\prod_{j\in J}(x_{j,0}-x_{j,1})^2.
		\end{align*}
		The products are taken in increasing order of $j$, and empty
		products are interpreted as $1$.
	\end{thm}
	\proof By the noncommutative Shephard-Todd-Chevalley theorem
	for skew polynomial algebras
	\cite[Theorem~1.1]{kkz2}, the group $\mc G$ is a reflection
	group for $\mc A$ if and only if it is generated by
	quasi-reflections. Under the isomorphism~\eqref{eq translation group},
	Lemma~\ref{lem ref} identifies the quasi-reflections
	in $\mc G$ with the nonzero translation vectors in $H$
	supported on a single component $G_j$ of order 2. 
	
	Suppose first that $\mc G$ is a reflection group for
	$\mc A$. Then $H$ is generated by the translation
	vectors corresponding to quasi-reflections.
	Each such vector has exactly one nonzero coordinate
	, and the corresponding
	group has order 2.  If $|G_j|>2$ for some $j$, none of these vectors can
	have a nonzero $j$-th coordinate. Hence every element
	of $H$ has zero $j$-th coordinate, so the projection
	of $H$ onto $G_j$ is zero. This contradicts the mesh
	condition $G_j=\langle d_{1,j},\ldots,d_{p,j}\rangle$,
	which ensures that this projection is surjective.
	Thus $|G_j|\le2$ for every $j$, proving condition~(1).
	
	For each $j\in J$, let $e_j\in P$ denote the vector
	whose $j$-th coordinate is the unique nonzero element
	of $G_j$ and whose other coordinates are zero.
	The translation vectors corresponding to quasi-reflections
	are precisely those $e_j$ that belong to $H$.
	Since these vectors generate $H$ and the projection
	$H\to G_j$ is surjective, $e_j$ must belong to $H$
	for every $j\in J$. As the remaining components of $P$
	are trivial, the vectors $\{e_j\mid j\in J\}$ generate
	$P$. Therefore $H=P$, proving condition~(2).
	
	Conversely, suppose that (1) and (2) hold.
	Then every $e_j$, $j\in J$, belongs to $H$, and these
	vectors generate $H=P$. 
	The elements $\{	\trv^{-1}(e_j), j\in J\}$
	are quasi-reflections by Lemma~\ref{lem ref}, and generate $\mc G$.
	Thus $\mc G$ is a reflection group for $\mc A$.
	
	Under condition (1), identify $P=\prod_{j=1}^pG_j$
	with the coordinate subspace
	\[
	\{(u_1,\ldots,u_p)\in\mathbb Z_2^p
	\mid u_j=0\text{ for }j\notin J\}.
	\]
	This space has dimension $q$, and $H$ is the row space
	of $D$. Hence $H=P$ if and only if
	$\operatorname{rank}_{\mathbb Z_2}D=q$.

	When these conditions hold, we also obtain
	$\mc G\cong H\cong(\mathbb Z_2)^q$. 	We next determine the invariant subalgebra.
	The degree-one basis is now
	\[
	Y=\{y_{i,0}\mid 1\le i\le p\}
	\cup\{y_{j,1}\mid j\in J\}.
	\]
	For each $j\in J$, let $h_j=\trv^{-1}(e_j)$. By~\eqref{eq translation eigenvalues}, $h_j$ fixes
	every element of $Y$ except $y_{j,1}$, and
	\[
	h_j(y_{j,1})=-y_{j,1}.
	\]
	Thus an ordered monomial in $Y$ is fixed by $h_j$
	if and only if its exponent of $y_{j,1}$ is even.
	Since these $h_j$ generate $\mc G$, comparison of
	coefficients in the ordered monomial basis of $\mc{A}$ shows that
	$\mc A^{\mc G}$ is generated as an algebra by
	\[
	\{y_{1,0},\cdots,y_{p,0}\}
	\cup
	\{y_{j,1}^2\mid j\in J\}.
	\]
	
	The defining relations of $\mc A$ show that the
	elements $y_{i,0}$ commute with one another.
	Moreover, all commutation coefficients among the
	elements of $Y$ are $\pm1$, so every $y_{j,1}^2$
	is central. Hence the displayed invariant generators
	commute pairwise.
	These generators are algebraically independent by
	the ordered monomial basis of $\mc A$. Hence
	\[
	\mc A^{\mc G}
	=\kk[y_{1,0},\ldots,y_{p,0},\
	y_{j,1}^2\ (j\in J)]
	\]
	is a commutative polynomial algebra.

	Finally, for each $j\in J$, define
	$\chi_j\in\widehat{\mc G}$ by
	\[
	\chi_j(h_k)=(-1)^{\delta_{jk}}
	\]
	for $k\in J$.
	Since the $h_j$ are independent generators of order 2,
	every character of $\mc G$ has a unique expression
	\[
	\chi=\prod_{j\in J}\chi_j^{\eta_j},
	\qquad \eta_j\in\{0,1\}.
	\]
	An ordered monomial belongs to $\mc A_\chi$
	if and only if the exponent of $y_{j,1}$ is congruent
	to $\eta_j$ modulo 2 for every $j\in J$.
	Set
	\[
	f_\chi=\prod_{j\in J}y_{j,1}^{\eta_j},
	\]
	with the factors arranged in increasing order of $j$. 
	Then 
	\[
	\mc A_\chi
	=\mc A^{\mc G}f_\chi
	=f_\chi\mc A^{\mc G}.
	\]
	These are free left and right modules of rank one,
	since $\mc A$ is a domain, so multiplication by the
	nonzero element $f_\chi$ on either side is injective.

	By Proposition~\ref{prop mesh hdet}, each $h_j$ has
	homological determinant $-1$. Therefore
	\[
	\hdet=\prod_{j\in J}\chi_j=\hdet^{-1}.
	\]
	Applying Lemma~\ref{lem j,a}, we conclude that both
	$\mathbf j_{\mc A,\mc G}$ and
	$\mathbf a_{\mc A,\mc G}$ are nonzero scalar
	multiples of
	\[
	\prod_{j\in J}y_{j,1}.
	\]
	Their product is consequently a nonzero scalar
	multiple of
	\[
	\left(\prod_{j\in J}y_{j,1}\right)^2
	=\pm\prod_{j\in J}y_{j,1}^2,
	\]
	where the sign arises from the skew commutation
	relations. Since
	\[
	y_{j,1}=x_{j,0}-x_{j,1},
	\]
	for $j\in J$, 	the asserted formulas follow.
	
	If $J=\varnothing$, then $H=\mc G=\{1\}$, and the
	same argument applies with all empty products equal
	to $1$.
	\qed

	\subsection{Noncommutative Auslander Theorem} In the previous sections we examined quasi-reflections and reflection groups in the context of distributive  YB  algebras. We now turn to a noncommutative analogue of the classical Auslander theorem.
	
	For each $1\le j\le p$, let $P_j$ denote the $j$-th
	coordinate subgroup of $P$, namely
	\[
	P_j=\{(u_1,\ldots,u_p)\in P
	\mid u_k=0\text{ for }k\ne j\}.
	\]
	
	\begin{thm}\label{thm mesh Auslander strengthened}
		The natural map
		\[
		\mc A\#\kk\mc G
		\rightarrow
		\operatorname{End}_{\mc A^{\mc G}}
		(\mc A_{\mc A^{\mc G}})\]
		is an isomorphism of algebras if and only if
		\[
		H\cap P_j=\{0\}
		\]
		for every $1\le j\le p$ with $|G_j|=2$.
	\end{thm}
	
	\proof
	Under the isomorphism~\eqref{eq translation group},
	Lemma~\ref{lem ref} identifies the reflections
	in $\mc G$ with the nonzero elements of $H\cap P_j$,
	where $j$ ranges over the indices with $|G_j|=2$.
	Consequently, $\mc G$ contains no reflections on
	$\mc A_1$ if and only if
	\[
	H\cap P_j=\{0\}\]
	for every $j$  with $|G_j|=2$.
	Since $\mc A$ is a skew polynomial algebra on
	$N\ge2$ generators and $\mc G$ is a finite group
	acting diagonally with respect to $Y$,
	the assertion follows from
	\cite[Theorem~5.5]{bhz1}.
	\qed
	
	In particular, the criterion above is automatically
	satisfied when none of the groups $G_j$ has order two.
	We therefore obtain the following corollary.
	
	\begin{cor}
		If $|G_j|\ne2$ for every $1\le j\le p$, then the
		natural map
		\[
		\mc A\#\kk\mc G
		\rightarrow
		\operatorname{End}_{\mc A^{\mc G}}
		(\mc A_{\mc A^{\mc G}})
		\]
		is an isomorphism of algebras.
	\end{cor}

	In \cite{bhz1}, the notion of pertinency was introduced to provide a computable
	criterion for the noncommutative Auslander theorem. For its formal definition,
	we refer to \cite[Definition 0.1]{bhz1}. In the following, we will compute a
	lower bound on the pertinency of the induced automorphism group on a
	distributive YB algebra. The main tool is the radical of group actions
	introduced in \cite{hz19}. We briefly recall the notions that will be used
	later.
	
	Let $G$ be a finite group acting on a Noetherian algebra $A$. Two sequences
	$(a_1,\dots,a_n)$ and $(b_1,\dots,b_n)$ of elements of $A$ are called
	\emph{pertinent} if
	\[
	\sum_{i=1}^{n}a_i\,(g\cdot b_i)=0\quad\text{for all }e\neq g\in G,
	\]
	We write
	$(a_1,\dots,a_n)\overset{G}{\sim}(b_1,\dots,b_n)$ for pertinent sequences. The \emph{radical} of the
	action is the ideal
	\[
	\mathfrak{r}(A,G)=\Big\{\,\sum\nolimits_{i=1}^{n}a_i b_i \;\Big|\;
	(a_1,\dots,a_n)\overset{G}{\sim}(b_1,\dots,b_n),\ n\ge 1\,\Big\}\subseteq A,
	\]
	and the quotient algebra
	\[
	\mathfrak{R}(A,G)=A/\mathfrak{r}(A,G)
	\]
	is called the \emph{pertinency algebra}. When $A$ has finite GK-dimension, the
	pertinency of the $G$-action on $A$ is
	\[
	\mathfrak{p}(A,G)=\GKdim(A)-\GKdim\big(\mathfrak{R}(A,G)\big),
	\]
	where $\GKdim$ denotes the GK-dimension.

	We now compute a lower bound for the pertinency of the permutation group action in a family of examples.

	Fix integers $p\ge2$ and $m\ge3$, and let $G_i=\mathbb Z_m$
	for all $1\le i\le p$. Take the $p\times p$ cyclic permutation matrix
	\[
	D=
	\begin{pmatrix}
		0 & 0 & \cdots & 0 & 1\\
		1 & 0 & \cdots & 0 & 0\\
		0 & 1 & \ddots & \vdots & \vdots\\
		\vdots & \ddots & \ddots & 0 & 0\\
		0 & \cdots & 0 & 1 & 0
	\end{pmatrix},\]
	That is, \[d_{i,j}=\begin{cases}
		1,&i=j+1\text{ or }(i,j)=(1,p),\\
		0,&\text{otherwise}.
	\end{cases}\]
	This gives an idempotent trivial affine mesh.

	The associated AS-regular algebra is isomorphic to 	$\mc{A}=T(V)/(\bold{R})$, where $V$ is the vector space with a basis $X=\{x_{i,a}| 1\leqslant i\leqslant p, a\in \ZZ_{m}\}$ and
	$$\bold{R}=\{x_{i,a}x_{j,b}-x_{j,b+d_{i,j}}x_{i,a-d_{j,i}}\mid 1\leqslant i\leqslant p,a, b\in \ZZ_{m}\}.$$
	
	Let $\xi$ be a primitive $m$-th root of unity. For each $1\le i\le p$, $a\in \ZZ_m$, set  $y_{i,a}=\sum_{b=0}^{m-1}\xi^{a\cdot b}x_{i,b}$. Following from Theorem \ref{thm}, we have the following equations
	\begin{equation}\label{e yy}
		y_{j,b}y_{i,a}=\xi^{b\cdot d_{i,j}-a\cdot d_{j,i}}y_{i,a}y_{j,b}.
	\end{equation}
	and
	\begin{equation}\label{e gy}
		g_i(y_{j,b})=\xi^{-b\cdot d_{i,j}}y_{j,b}.
	\end{equation}
	Then the permutation group ${\mc{G}}$ is isomorphic to the group $$\langle g_1,g_2,\cdots,g_p\mid g_i^m=e, g_ig_j=g_jg_i, 1\le i,j\le p\rangle.$$
	That is, 		$\mc G\cong(\mathbb Z_m)^p$.

	\begin{lem}\label{L0}
		If  $\gcd(a_i,m)=1$ for each $1\leqslant i\leqslant p$, then    $y_{1,a_1}^{m}y_{2,a_2}^{m}\cdots y_{p,a_p}^{m}\in \mk{r}(\mc{A},\mc{G})$.
	\end{lem}
	\proof
	Set $z_i=y_{i,a_i}$ and $\mathcal I=\{0,\ldots,m-1\}^p$.
	For $\mathbf r=(r_1,\ldots,r_p)\in\mathcal I$, define
	\[
	U_{\mathbf r}=z_p^{m-r_p}\cdots z_1^{m-r_1} 
	\text{ and } W_{\mathbf r}=z_1^{r_1}\cdots z_p^{r_p}.
	\]
	Every $z_i^m$ is central by~\eqref{e yy}. Therefore,
	\[
	U_{\mathbf r}W_{\mathbf r}=z_1^m\cdots z_p^m.
	\]
	For $g=g_1^{\alpha_1}\cdots g_p^{\alpha_p}$, $0\le\alpha_i<m$, set
	$\alpha_{p+1}=\alpha_1$. Equation~\eqref{e gy} gives
	\[
	g(W_{\mathbf r})
	=\xi^{-\sum_{i=1}^p\alpha_{i+1}a_ir_i}W_{\mathbf r}.
	\]
	Hence
	\[
	\sum_{\mathbf r\in\mathcal I}U_{\mathbf r}g(W_{\mathbf r})
	=\left(\prod_{i=1}^p
	\sum_{r_i=0}^{m-1}\xi^{-\alpha_{i+1}a_ir_i}\right)
	z_1^m\cdots z_p^m.
	\]
	If $g\ne e$, then some $\alpha_{i+1}$ is nonzero modulo $m$.
	Since $a_i$ is a unit modulo $m$, the corresponding geometric
	sum vanishes. Thus the two sequences are pertinent.  Since
	\[
	\sum_{\mathbf r\in\mathcal I}U_{\mathbf r}W_{\mathbf r}
	=m^p z_1^m\cdots z_p^m
	\]
	and  $m^p$ is invertible in $\kk$, the assertion follows.
	\qed

	\begin{thm}\label{thm cyclic pertinency}
		For the cyclic mesh above,
		\[
		\mathfrak p(\mc A,\mc G)\ge\varphi(m),
		\]
		where $\varphi$ is Euler's totient function.
	\end{thm}
	
	\proof
	Set $u_{i,a}=y_{i,a}^m$ and
	\[
	\mc B=\kk[u_{i,a}\mid 1\le i\le p,\ 0\le a<m].
	\]
	By~\eqref{e yy}, $\mc B$ is central in $\mc A$.
	The PBW basis of $\mc A$ shows that $\mc B$ is a polynomial
	algebra in $pm$ variables and that $\mc A$ is finitely generated as 
	$\mc B$-module. It is spanned over $\mc B$ by the ordered
	monomials whose exponents are all less than $m$.
	Let
	\[
	\mc B'=\mc B/(\mc B\cap\mathfrak r(\mc A,\mc G)). 
	\]
	Then $\mc B'$ is a central subalgebra of	the algebra $\mathfrak R(\mc A,\mc G)$,  and $\mathfrak R(\mc A,\mc G)$ is finitely generated as  $\mc B'$-module. Thus
	\[
	\operatorname{GKdim}\mathfrak R(\mc A,\mc G)
	=\operatorname{GKdim}\mc B'.
	\]
	For every $a\in\mathbb Z_m^\times$, Lemma~\ref{L0} gives
	\[
	M_a=u_{1,a}\cdots u_{p,a}
	\in\mc B\cap\mathfrak r(\mc A,\mc G).
	\]
	Let $\mathfrak a=(M_a\mid a\in\mathbb Z_m^\times)\subseteq\mc B$.
	The supports of these $\varphi(m)$ monomials are pairwise
	disjoint. Every prime ideal containing $\mathfrak a$ must therefore
	contain at least one variable from each support, so its
	height is at least $\varphi(m)$. Conversely, the prime ideal
	$(u_{1,a}\mid a\in\mathbb Z_m^\times)$ contains $\mathfrak a$ and has
	height $\varphi(m)$. Hence
	\[
	\dim(\mc B/\mathfrak a)=pm-\varphi(m).
	\]
	Since $\mc B'$ is a quotient of $\mc B/\mathfrak a$, we obtain
	\[
	\operatorname{GKdim}\mathfrak R(\mc A,\mc G)
	=\operatorname{GKdim}\mc B'
	\le pm-\varphi(m).
	\]
	Finally, with $\operatorname{GKdim}\mc A=pm$,
	$$\mk{p}(\mc{A},{\mc{G}})=pm-\GKdim(\mk{R}(\mc{A},{\mc{G}}))\ge \varphi(m).$$
	The desired bound holds.
	\qed

\subsection*{Declaration of generative AI and AI-assisted technologies in the manuscript preparation process} The first version of this paper was completed entirely by the authors. In preparing the revision, the authors used ChatGPT (model GPT-6 Sol, OpenAI) for grammar checking and language polishing, as well as for suggestions on describing the permutation group in Section 3 as a group of translations and on using reversed-order products in the pertinent sequences for the pertinency lower bound. Building on these suggestions, the authors improved the proofs and simplified the computations. All AI-assisted content was verified by the authors, who take full responsibility for the entire work.
	
		\subsection*{Acknowledgement}The  authors are supported by a grant from NSFC (No. 12371017).

	\vspace{5mm}
	
	\bibliography{}

\begin{thebibliography}{99}
		
		\bibitem{aus62} M. Auslander, \textit{On the purity of the branch locus}, Amer. J. Math. \textbf{84} (1962), no. 1, 116--125.
		
		\bibitem{bhz} Y.-H. Bao, J.-W. He, and J. J. Zhang, \textit{Pertinency of Hopf actions and quotient categories of Cohen-Macaulay algebras}, J. Noncommut. Geom. \textbf{13} (2019), no. 2, 667--710.
		
		\bibitem{bhz1} Y.-H. Bao, J.-W. He, and J. J. Zhang, \textit{Noncommutative Auslander theorem}, Trans. Amer. Math. Soc. \textbf{370} (2018), no. 12, 8613--8638.
		
		\bibitem{bax} R. J. Baxter, \textit{Partition function of the eight-vertex lattice model}, Ann. Physics \textbf{70} (1972), 193--228.
		
		\bibitem{br} E. Brieskorn, \textit{Automorphic sets and braids and singularities}, in \textit{Braids} (Santa Cruz, CA, 1986), Contemp. Math. \textbf{78} (1988), 45--115.
		
		\bibitem{bg1} K. A. Brown and K. R. Goodearl, \textit{Lectures on algebraic quantum groups}, Birkh\"{a}user Verlag, Basel, 2002.
		
		\bibitem{cjks} J. S. Carter, D. Jelsovsky, S. Kamada, L. Langford, and M. Saito, \textit{Quandle cohomology and state-sum invariants of knotted curves and surfaces}, Trans. Amer. Math. Soc. \textbf{355} (2003), no. 10, 3947--3989.
		
		\bibitem{cho} F. Chouraqui, \textit{Garside groups and Yang-Baxter equation}, Comm. Algebra \textbf{38} (2010), no. 12, 4441--4460.
		
		\bibitem{de} P. Dehornoy, \textit{Set-theoretic solutions of the Yang-Baxter equation, RC-calculus, and Garside germs}, Adv. Math. \textbf{282} (2015), 93--127.
		
		\bibitem{dr92} V. G. Drinfeld, \textit{On some unsolved problems in quantum group theory}, in Quantum Groups (Leningrad, 1990), Lecture Notes in Math. \textbf{1510}, Springer, 1992, 1--8.
		
		\bibitem{ess} P. Etingof, T. Schedler, and A. Soloviev, \textit{Set-theoretical solutions to the quantum Yang-Baxter equation}, Duke Math. J. \textbf{100} (1999), 169--209.
		
		\bibitem{fr} R. Fenn and C. Rourke, \textit{Racks and links in codimension two}, J. Knot Theory Ramifications \textbf{1} (1992), no. 4, 343--406.
		
		\bibitem{gi4} T. Gateva-Ivanova, \textit{Skew polynomial rings with binomial relations}, J. Algebra \textbf{185} (1996), 710--753.
		
		\bibitem{gi6} T. Gateva-Ivanova, \textit{A combinatorial approach to the set-theoretic solutions of the Yang-Baxter equation}, J. Math. Phys. \textbf{45} (2004), no. 10, 3828--3858.
		
		\bibitem{gi1} T. Gateva-Ivanova, \textit{Binomial skew polynomial rings, Artin-Schelter regularity, and binomial solutions of the Yang-Baxter equation}, Serdica Math. J. \textbf{30} (2004), no. 2--3, 431--470.
		
		\bibitem{gi5} T. Gateva-Ivanova, \textit{Garside structures on monoids with quadratic square-free relations}, Algebr. Represent. Theory \textbf{14} (2011), 779--802.
		
		\bibitem{gi2} T. Gateva-Ivanova, \textit{Quadratic algebras, Yang-Baxter equation, and Artin-Schelter regularity}, Adv. Math. \textbf{230} (2012), 2152--2175.
		
		\bibitem{gi18} T. Gateva-Ivanova, \textit{Set-theoretic solutions of the Yang-Baxter equation, braces and symmetric groups}, Adv. Math. \textbf{328} (2018), 649--701.
		
		\bibitem{gi} T. Gateva-Ivanova, \textit{Veronese subalgebras and Veronese morphisms for a class of Yang-Baxter algebras}, J. Noncommut. Geom. \textbf{20} (2026), 221--267.
		
		
		\bibitem{gim1} T. Gateva-Ivanova and S. Majid, \textit{Matched pairs approach to set theoretic solutions of the Yang-Baxter equation}, J. Algebra \textbf{319} (2008), 1462--1529.
		
		\bibitem{gim2} T. Gateva-Ivanova and S. Majid, \textit{Quantum spaces associated to multipermutation solutions of level two}, Algebr. Represent. Theory \textbf{14} (2011), 341--376.
		
		\bibitem{giv} T. Gateva-Ivanova and M. Van den Bergh, \textit{Semigroups of I-type}, J. Algebra \textbf{206} (1998), 97--112.
		
		\bibitem{hz19} J.-W. He and Y. Zhang, \textit{Local cohomology associated to the radical of a group action on a noetherian algebra}, Israel J. Math. \textbf{231} (2019), 303--342.
		
		\bibitem{jpz} P. Jedli\u{a}ka, A. Pilitowska, and A. Zamojska-Dzienio, \textit{The construction of multipermutation solutions of the Yang-Baxter equation of level 2}, J. Combin. Theory Ser. A \textbf{176} (2020), 105295, 35 pp.
		
		\bibitem{jkv} E. Jespers, {\L}. Kubat, and A. Van Antwerpen, \textit{The structure monoid and algebra of a non-degenerate set-theoretic solution of the Yang-Baxter equation}, Trans. Amer. Math. Soc. \textbf{372} (2019), no. 10, 7191--7223.
		
		\bibitem{joz} P. J{\o}rgensen and J. J. Zhang, \textit{Gourmet's guide to Gorensteinness}, Adv. Math. \textbf{151} (2000), 313--345.
		
		\bibitem{kas} C. Kassel, \textit{Quantum Groups}, Graduate Texts in Math. \textbf{155}, Springer-Verlag, 1995.
		
		\bibitem{kkz1} E. Kirkman, J. Kuzmanovich, and J. J. Zhang, \textit{Rigidity of graded regular algebras}, Trans. Amer. Math. Soc. \textbf{360} (2008), 6331--6369.
		\bibitem{kkz2}
		E. Kirkman, J. Kuzmanovich, and J. J. Zhang,
		\textit{Shephard-Todd-Chevalley theorem for skew polynomial rings},
		Algebr. Represent. Theory \textbf{13} (2010), 127--158.
		\bibitem{kz} E. Kirkman and J. J. Zhang, \textit{The Jacobian, reflection arrangement and discriminant for reflection Hopf algebras}, Int. Math. Res. Not. \textbf{2021} (2021), no. 13, 9853--9907.
		
		\bibitem{lyz} J.-H. Lu, M. Yan, and Y.-C. Zhu, \textit{On the set-theoretical Yang-Baxter equation}, Duke Math. J. \textbf{104} (2000), no. 1, 153--170.
		
		\bibitem{man} Y. I. Manin, \textit{Quantum groups and noncommutative geometry}, Universit\'{e} de Montr\'{e}al Centre de Recherches Math\'{e}matiques, Montreal, QC, 1988.
		
		\bibitem{ms} I. Mori and S. P. Smith, \textit{m-Koszul Artin-Schelter regular algebras}, J. Algebra \textbf{446} (2016), 373--399.
		
		\bibitem{rrz} M. Reyes, D. Rogalski, and J. J. Zhang, \textit{Skew Calabi-Yau algebras and homological identities}, Adv. Math. \textbf{264} (2014), 308--354.
		
		\bibitem{ru} W. Rump, \textit{A decomposition theorem for square-free unitary solutions of the quantum Yang-Baxter equation}, Adv. Math. \textbf{193} (2005), no. 1, 40--55.
		
		\bibitem{ru07} W. Rump, \textit{Braces, radical rings, and the quantum Yang-Baxter equation}, J. Algebra \textbf{307} (2007), no. 1, 153--170.
		
		\bibitem{wz} Q. S. Wu and C. Zhu, \textit{Skew group algebras of Calabi-Yau algebras}, J. Algebra \textbf{340} (2011), 53--76.
		
		\bibitem{yang} C. N. Yang, \textit{Some exact results for the many-body problem in one dimension with repulsive delta-function interaction}, Phys. Rev. Lett. \textbf{19} (1967), 1312--1315.
		
	\end{thebibliography}

\end{document}